\documentclass[12pt,a4paper]{article}
\usepackage{amsmath}
\usepackage{amssymb}
\usepackage{amsxtra}
\usepackage{amscd}
\usepackage{amsthm}
\usepackage{hyperref}
\usepackage[mathscr]{eucal} 

\theoremstyle{plain}

\newtheorem{sbthm}[subsubsection]{Theorem}
\newtheorem{sbprop}[subsubsection]{Proposition}
\newtheorem{sbcor}[subsubsection]{Corollary}
\newtheorem{sblem}[subsubsection]{Lemma}

\theoremstyle{definition}

\newtheorem{sbrem}[subsubsection]{Remark}

\newtheorem{sbpara}[subsubsection]{}

\newenvironment{pf}{\proof[\proofname]}{\endproof}
\begin{document}

\title{Refined Swan conductors mod $p$ of one-dimensional Galois representations}

\author
{Kazuya Kato, Isabel Leal, Takeshi Saito}

\maketitle

\newcommand{\lr}[1]{\langle#1\rangle}
\newcommand{\ul}[1]{\underline{#1}}
\newcommand{\eq}[2]{\begin{equation}\label{#1}#2 \end{equation}}
\newcommand{\ml}[2]{\begin{multline}\label{#1}#2 \end{multline}}
\newcommand{\ga}[2]{\begin{gather}\label{#1}#2 \end{gather}}
\newcommand{\mc}{\mathcal}
\newcommand{\mb}{\mathbb}
\newcommand{\surj}{\twoheadrightarrow}
\newcommand{\inj}{\hookrightarrow}
\newcommand{\red}{{\rm red}}
\newcommand{\codim}{{\rm codim}}
\newcommand{\rank}{{\rm rank}}
\newcommand{\Pic}{{\rm Pic}}
\newcommand{\Div}{{\rm Div}}
\newcommand{\Res}{{\rm Res}}
\newcommand{\Sw}{{\rm Sw}}
\newcommand{\Rsw}{{\rm Rsw}}
\newcommand{\divi}{{\rm div}}
\newcommand{\Hom}{{\rm Hom}}
\newcommand{\Ext}{{\rm Ext}}
\newcommand{\im}{{\rm im}}
\newcommand{\fil}{{\rm fil}}
\newcommand{\gp}{{\rm gp}}
\newcommand{\Spec}{{\rm Spec}}
\newcommand{\Sing}{{\rm Sing}}
\newcommand{\Char}{{\rm char}}
\newcommand{\ab}{{\rm ab}}
\newcommand{\Tr}{{\rm Tr}}
\newcommand{\Gal}{{\rm Gal}}
\newcommand{\Min}{{\rm Min}}
\newcommand{\mult}{{\rm mult}}
\newcommand{\Max}{{\rm Max}}
\newcommand{\Alb}{{\rm Alb}}
\newcommand{\gr}{{\rm gr}}
\newcommand{\Ker}{{\rm Ker}}
\newcommand{\Lie}{{\rm Lie}}
\newcommand{\infi}{{\rm inf}}
\newcommand{\pole}{{\rm pole}}
\newcommand{\ti}{\times }
\newcommand{\modu}{{\rm mod}}
\newcommand{\GL}{{\rm GL}}
\newcommand{\Ab}[1]{{\mathcal A} {\mathit b}/#1}
\newcommand{\sA}{{\mathcal A}}
\newcommand{\sB}{{\mathcal B}}
\newcommand{\sC}{{\mathcal C}}
\newcommand{\sD}{{\mathcal D}}
\newcommand{\sE}{{\mathcal E}}
\newcommand{\sF}{{\mathcal F}}
\newcommand{\sG}{{\mathcal G}}
\newcommand{\sH}{{\mathcal H}}
\newcommand{\sI}{{\mathcal I}}
\newcommand{\sJ}{{\mathcal J}}
\newcommand{\sK}{{\mathcal K}}
\newcommand{\sL}{{\mathcal L}}
\newcommand{\sM}{{\mathcal M}}
\newcommand{\sN}{{\mathcal N}}
\newcommand{\sO}{{\mathcal O}}
\newcommand{\sP}{{\mathcal P}}
\newcommand{\sQ}{{\mathcal Q}}
\newcommand{\sR}{{\mathcal R}}
\newcommand{\sS}{{\mathcal S}}
\newcommand{\sT}{{\mathcal T}}
\newcommand{\sU}{{\mathcal U}}
\newcommand{\sV}{{\mathcal V}}
\newcommand{\sW}{{\mathcal W}}
\newcommand{\sX}{{\mathcal X}}
\newcommand{\sY}{{\mathcal Y}}
\newcommand{\sZ}{{\mathcal Z}}
\newcommand{\A}{{\mathbb A}}
\newcommand{\B}{{\mathbb B}}
\newcommand{\C}{{\mathbb C}}
\newcommand{\D}{{\mathbb D}}
\newcommand{\E}{{\mathbb E}}
\newcommand{\F}{{\mathbb F}}
\newcommand{\G}{{\mathbb G}}
\renewcommand{\H}{{\mathbb H}}
\newcommand{\I}{{\mathbb I}}
\newcommand{\J}{{\mathbb J}}
\newcommand{\M}{{\mathbb M}}
\newcommand{\N}{{\mathbb N}}
\renewcommand{\P}{{\mathbb P}}
\newcommand{\Q}{{\mathbb Q}}
\newcommand{\R}{{\mathbb R}}
\newcommand{\T}{{\mathbb T}}
\newcommand{\U}{{\mathbb U}}
\newcommand{\V}{{\mathbb V}}
\newcommand{\W}{{\mathbb W}}
\newcommand{\X}{{\mathbb X}}
\newcommand{\Y}{{\mathbb Y}}
\newcommand{\Z}{{\mathbb Z}}
\newcommand{\pic}{{\text{Pic}(C,\sD)[E,\nabla]}}
\newcommand{\ocd}{{\Omega^1_C\{\sD\}}}
\newcommand{\oc}{{\Omega^1_C}}
\newcommand{\al}{{\alpha}}
\newcommand{\be}{{\beta}}
\newcommand{\la}{{\lambda}}
\newcommand{\ta}{{\theta}}
\newcommand{\ve}{{\varepsilon}}
\newcommand{\phe}{{\varphi}}
\newcommand{\om}{{\overline M}}
\newcommand{\sym}{{\text{Sym}(\om)}}
\newcommand{\an}{{\text{an}}}
\newcommand{\bs}{{\backslash}}
\newcommand{\lra}{\longrightarrow}
\newcommand{\lam}{{\lambda}}
\newcommand{\sig}{{\sigma}}
\newcommand{\cF}{{\cal F}}
\newcommand{\cO}{{\cal O}}
\newcommand{\cG}{{\cal G}}
\newcommand{\cC}{{\cal C}}
\newcommand{\cS}{{\cal S}}
\newcommand{\cT}{{\cal T}}
\newcommand{\cH}{{\cal H}}


Dedicated to Shuji Saito on his 60th birthday.

\section{Introduction}

Let $K$ be a complete discrete valuation field whose residue field $F$ is of characteristic $p>0$. 
 The Swan conductor 
 \begin{equation}
 \Sw(\chi)\in \Z_{\geq 0}\;\;\text{for}\;\; \chi\in H^1(K, \Q_p/\Z_p)=\Hom_{\text{cont}}(\Gal(\bar K/K), \Q_p/\Z_p)
 \label{eqSw}
 \end{equation} 
  generalizing the classical perfect residue field case, the subgroups 
 \begin{equation}
 F_nH^1(K, \Q_p/\Z_p)= \{\chi\in 
H^1(K, \Q_p/\Z_p)\;|\; \Sw(\chi)\leq n\} 
\end{equation}
of $H^1(K, \Q_p/\Z_p)$ for $n\geq 0$, and an injective homomorphism
\begin{equation} \text{rsw} \;:\; F_nH^1(K, \Q_p/\Z_p)/F_{n-1}H^1(K, \Q_p/\Z_p) \to {\frak m}_K^{-n}/{\frak m}_K^{-n+1} \otimes_{O_K} \Omega^1_{O_K}(\log)
\label{eqrsw}
\end{equation}
called the refined Swan conductor for $n\geq 1$, where $\Omega^1_{O_K}(\log)$ is the module of differential forms with log poles, are defined in \cite{KKs}. Let $m=\max(n-e_K, [n/p])$ where $e_K$ denotes the absolute ramification index $\text{ord}_K(p)$ of $K$ ($e_K=\infty$ if $K$ is of characteristic $p$) and $[n/p]= \max\{x\in \Z\;|\; x\leq n/p\}$. In this paper, we define an injective  homomorphism
\begin{equation}
\Rsw\;:\; F_nH^1(K, \Q_p/\Z_p)/F_mH^1(K, \Q_p/\Z_p)  \to {\frak m}_K^{-n}/{\frak m}_K^{-m} \otimes_{O_K} \Omega^1_{O_K}(\log)
\label{eqRsw}
\end{equation}
which is a lifting of (\ref{eqrsw}).
Note that ${\frak m}_K^{-m}=p{\frak m}^{-n}_K+{\frak m}_K^{-[n/p]}$. We will call the homomorphism  (\ref{eqrsw}) the refined Swan conductor mod ${\frak m}_K$ and the homomorphism (\ref{eqRsw})  the refined Swan conductor mod $p$. 

In the case $K$ is of characteristic $p$ (in this case, $m=[n/p]$), the homomorphism Rsw is defined by using Artin-Schreier-Witt theory and is already known (\cite{Ma}, \cite{Bo2}). In the mixed characteristic case,  we define Rsw by using higher dimensional class field theory. See Section 3 for the definition of Rsw.

In Section 4, we show that for a regular scheme $X$ of finite type over $\Z$ and for a divisor $D$ on $X$ with simple normal crossings, and for $U=X\smallsetminus D$ and $\chi\in H^1_{et}(U, \Q_p/\Z_p)$, the refined Swan conductors mod $p$ of $\chi$ at generic points of $D$ defined in Section 3 glue to a global section of a certain sheaf of differential forms on $X$. 

We will give applications \ref{thmB}, \ref{thmIL1} and \ref{thmC} of our theory. \ref{thmB} (resp. \ref{thmC}) shows that the Swan conductor (\ref{eqSw}) is recovered from pullbacks of $\chi$ to the perfect residue field cases (resp. to one dimensional subschemes in the situation of Section 4). \ref{thmIL1} improves a result of the second author in \cite{IL} concerning the change of the Swan conductor (\ref{eqSw}) in a transcendental extension of local fields.

Shuji Saito played leading roles in the developments of higher dimensional class field theory and its applications. We dedicate this paper to him  with admiration. 

\section{Preliminaries}

\subsection{Differential forms and discrete valuations}\label{difdv}

\begin{sbpara}
Let $K$ be a discrete valuation field with residue field $F$,  valuation ring $O_K$ and maximal ideal ${\frak m}_K$. Assume $F$ is  of characteristic $p>0$. 

Let $e_K=\text{ord}_K(p)$ be the absolute ramification index of $K$ ($e_K=\infty$ if $K$ is of characteristic $p$). 

\end{sbpara}

\begin{sbpara}

Let $\Omega^1_{O_K}(\log) =\Omega^1_{O_K/\Z}(\log)$ be the module of differential forms with log poles of $O_K$ over $\Z$ with respect to the standard log structure of $O_K$ (\cite{KK2}). It is the $O_K$-module defined by the following generators and relations.

Generators. $dx$ for $x\in O_K$ and $d\log(x)$ for $x\in K^\times$.

Relations. $d(x+y)= dx+dy$ and $d(xy)= xdy+ydx$ for $x,y\in O_K$. $d\log(xy)=d\log(x)+d\log(y)$ for $x,y\in K^\times$. $dx=xd\log(x)$ for $x\in O_K\smallsetminus \{0\}$.

Let $${\hat \Omega}^1_{O_K}(\log)=\varprojlim_n \Omega^1_{O_K}(\log)/{\frak m}_K^n\Omega^1_{O_K}(\log).$$ 
Note that ${\hat \Omega}^1_{O_K}(\log)/{\frak m}^n_K{\hat \Omega}^1_{O_K}(\log)=\Omega^1_{O_K}(\log)/{\frak m}_K^n\Omega^1_{O_K}(\log)$ for any $n\geq 1$ (\cite{FK} Chap. 0, 7.2.8, 7.2.16). 

\end{sbpara}

The following 2 of \ref{hatdif} is the log version of \cite{Ku0} Lemma 1.1 and the proof given below is essentially the same as that in loc. cit. 
\begin{sblem}\label{hatdif}  Let $(b_{\lam})_{\lam\in \Lambda}$ be a $p$-base of $F$ (\cite{Gr}  21.1.9) and let $\tilde b_{\lam}$ is a lifting of $b_{\lam}$ to $O_K$ for each $\lam$. 

1. Assume $K$ is of characteristic $p$, and let $\pi$ be a prime element of $K$. Then ${\hat \Omega}^1_{O_K}(\log)$ is the ${\frak m}_K$-adic completion of the free $O_K$-module with base $d\tilde b_{\lam}$ ($\lam\in \Lambda$) and $d\log(\pi)$. If $K$ is complete and $[F:F^p]<\infty$, we have ${\hat \Omega}^1_{O_K}(\log)=\Omega^1_{O_K}(\log)$. 

2. If $K$ is of mixed characteristic,  we have an isomorphism of topological $O_K$-modules
$${\hat \Omega}^1_{O_K}(\log)\cong ({\hat \oplus}_{\Lambda} O_K) \oplus O_K/{\frak m}_K^a,$$
for some integer $a\geq 1$, where ${\hat \oplus}_{\Lambda} O_K$ denotes the ${\frak m}_K$-adic completion of the free $O_K$-module $\oplus_{\Lambda} O_K$ with base $\Lambda$, such that the $\lam$-th base in $\oplus_{\Lambda} O_K$ is sent to $d\tilde b_{\lam}\in {\hat \Omega}^1_{O_K}(\log)$. 

3. If $K$ is of mixed characteristic, 
$O_K/pO_K\otimes_{O_K} \Omega^1_{O_K}(\log)$ is a free $O_K/pO_K$-module with base $d\tilde b_{\lam}$ ($\lam\in \Lambda$) and $d\log(\pi)$. 

\end{sblem}
\begin{pf} We may assume that $K$ is complete (and so we assume it).

1 follows from the fact that $O_K\cong F[[T]]$ in this case. 

Assume $K$ is of characteristic $0$. Then $K$ is a finite extension of a complete discrete valuation field $K_0$ such that $p$ is a prime element in $K_0$ and the residue field of $K_0$ coincides with that of $K$. We may assume $\tilde b_{\lam}\in O_{K_0}$ for all $\lam$. As is easily seen,  there is an isomorphism between the  ${\frak m}_{K_0}$-adic completion ${\hat \Omega}^1_{O_{K_0}}$ of the differential module $\Omega^1_{O_{K_0}/\Z}$ (without log) and ${\hat \oplus}_{\Lambda} O_{K_0}$ which sends the $\lam$-th base of $\oplus_{\Lambda} O_{K_0}$ to $d\tilde b_{\lam}$. Let $\pi$ be a prime element of $K$ and let $f(T)=\sum_{i=0}^e a_iT^i$ be the irreducible polynomial of $\pi$ over $K$ such that $a_e=1$. Then $a_0$ is a prime element of $K_0$ and $p|a_i$ for all $0\leq i<e$. We have $O_{K_0}[T]/(f(T))\overset{\cong}\to O_K\;;\;T\mapsto \pi$. From this, we see that $\hat \Omega^1_{O_K}(\log)$ has the presentation with generators the $O_K$-module $O_K\otimes_{O_{K_0}} {\hat \Omega}^1_{O_{K_0}}$ and $d\log(\pi)$ and with the relation $f'(\pi)\pi d\log(\pi)+\sum_{i=0}^{e-1} \pi^ida_i=0$. Here $f'(\pi)\pi\in {\frak m}_K\smallsetminus \{0\}$. This proves 2. 

We prove 3. The last relation is written as 
$$(\sum_{i=1}^e ia_i \pi^i d\log(\pi)) + \sum_{i=0}^{e-1} (\pi^ia_0d(a_i/a_0) + \pi^ia_i d\log(a_0))=0$$
which is trivial mod $p$. 
\end{pf}

\begin{sbrem} Note that the version of 2 of \ref{hatdif}  without log poles is false. For example, if $F$ is perfect and $p$ is a prime element of $K$,
 $O_K/pO_K \otimes_{O_K} \Omega^1_{O_K}=0$ whereas $O_K/pO_K\otimes_{O_K} \Omega^1_{O_K}(\log)\cong O_K/pO_K$ with base $d\log(p)$. Our theory will go well with log poles. 

\end{sbrem}

\begin{sblem}\label{KLdif} Let $L/K$ be a separable extension of complete discrete valuation fields. Assume that the residue field of $K$ is perfect.  Then the map
$O_L \otimes_{O_K} {\hat \Omega}^1_{O_K}(\log)\to {\hat \Omega}^1_{O_L}(\log)$ is injective. 
\end{sblem}

\begin{pf} Assume first $K$ is of characteristic $p$. Let $\pi_K$ be a prime element of $K$. By 1 of \ref{hatdif}, $\hat \Omega^1_{O_K}(\log)$ is a free $O_K$-module of rank $1$ with base $d\log(\pi_K)$.  As is easily seen, if $f\in O_L$ and $df=0$ in $\hat \Omega^1_{O_L}(\log)$, then $f$ is a $p$-th power in $O_L$. Hence by the assumption that $L/K$ is separable, we have $d\pi_K\neq 0$ in ${\hat \Omega}^1_{O_L}(\log)$. By \ref{hatdif}, this proves the injectivity in question. 

Assume next $K$ is of characteristic $0$.  We have $K\subset L'\subset L$ where $L'$ is a complete discrete valuation field such that $L$ is a finite extension of $L'$, the residue field of $L'$ coincides with that of $L$, and a prime element of $K$ is still a prime element of $L'$. Take a $p$-base $(b_{\lam})_{\lam\in \Lambda}$ of the residue field of $L$ and let $\tilde b_{\lam}$ be a lifting of $b_{\lam}$ to $O_{L'}$ for each $\lam$. Then the proof of 2 of \ref{hatdif} shows that ${\hat \Omega}^1_{O_{L'}}(\log)= O_{L'} \otimes_{O_K} {\hat \Omega}^1_{O_K}(\log) \oplus ({\hat \oplus}_{\Lambda} O_{L'})$ where the $\lam$-th base of $\oplus_{\Lambda} O_{L'}$ is sent to $d\tilde b_{\lam}\in {\hat \Omega}^1_{O_{L'}}(\log)$. Let $\pi$ be a prime element of $L$ and let $f(T)=\sum_{i=0}^e a_iT^i$ be the irreducible polynomial of $\pi$ over $L'$ with $a_e=1$. Then $a_0$ is a prime element of $L'$ and $a_0|a_i$ for $0\leq i\leq e-1$. From the fact $O_{L'}[T]/(f(T))\overset{\cong}\to O_L\;;\;T\mapsto \pi$, we have that ${\hat \Omega}^1_{O_L}(\log)$ has a presentation with the generators the $O_L$-module $O_L\otimes_{O_{L'}} {\hat \Omega}^1_{O_{L'}}(\log)$ and $d\log(\pi)$ and with the relation 
$$(f'(\pi)\pi/a_0)d\log(\pi) + \sum_{i=0}^{e-1} \pi^i (a_i d\log(a_0) + a_0 d(a_i/a_0))=0$$
(note $f'(\pi)\pi/a_0\in O_L\smallsetminus \{0\}$). This proves \ref{KLdif}. 
\end{pf}

\begin{sbpara} Let $a$ be an integer such that $0\leq a\leq e_K$. Take a ring homomorphism $\iota: F\to O_K/{\frak m}_K^a$ which lifts the identity map of $F$ ($\iota$ exists by a theorem of Cohen (\cite{Gr} 19.6.1) and a prime element $\pi$ of $K$. Then we  have an isomorphism 
\begin{equation}
F[T]/(T^a) \overset{\cong}\to O_K/{\frak m}^a_K\;;\;\sum_i a_iT^i\mapsto \sum_i \iota(a_i)\pi^i
\label{isoTi}
\end{equation}
 of rings with log structures given by $T$ and $\pi$, respectively, which sends $T$ to $\pi$ in the log structures. (See \cite{KK2} for log structures.)  Since $O_K/{\frak m}^a_K\otimes_{O_K} \Omega^1_{O_K}(\log)$ coincides with the module of differential forms with log poles (\cite{KK2}) of the ring $O_K/{\frak m}^a_K$ with the log structure given by $\pi$, it is isomorphic to the module of differential forms with log poles of the ring $F[T]/(T^a)$ with the log structure given by $T$. In the case $K$ is of mixed characteristic and $a=e_K$, this gives another proof of 3 of \ref{hatdif}. 
\end{sbpara}

In the rest of Section 2.1,  we consider the case $[F:F^p]=p^r<\infty$ and consider a  residue map whose target is the one dimensional $F$-vector space $\Omega^r_F=\wedge^r_F \Omega^1_F$.

\begin{sbprop}\label{Res1} Assume $K$ is complete and of characteristic $p$, and that $[F:F^p]=p^r<\infty$.  Consider a prime element $\pi$ of $K$, a ring homomorphism $\iota: F \to O_K$ which lifts the identity map of $F$,   and  the residue map  $$\Res\;:\;\Omega_K^{r+1}\to \Omega^r_F\; ;\; \sum_{i\gg -\infty}  \pi^i \iota(\omega_i) \wedge  d\log(\pi) \mapsto \omega_0 \quad (\omega_i\in \Omega^r_F).$$ Let $C: \Omega^r_F\to \Omega^r_F$ be the Cartier operator (\cite{IL} Chap. 0, Section 2). Then for integers $a, b$ such that $a\geq 1$ and $p^b\geq a$, the restriction of $C^b\circ \Res$ to $ {\frak m}_K^{1-a} \Omega^{r+1}_{O_K}(\log)$ is independent of the choices of $\pi$ and $\iota$.

\end{sbprop}

\begin{pf} Note that the Cartier operator $C: \Omega_F^r\to \Omega^r_F$ is characterized by the properties $C(x^p d\log y_1\wedge \dots\wedge  d\log y_{r})= x d\log y_1\wedge \dots \wedge d\log y_r$ for $x\in F$ and $y_i\in F^\times$ and $C(d\omega)=0 \;\;\text{for}\;\; \omega\in \Omega^{r-1}_F$. Using also the Cartier operator $C: \Omega_K^{r+1}\to \Omega^{r+1}_K$, we have $$\Res\circ C = C\circ \Res: \Omega^{r+1}_K \to \Omega^r_F.$$ We have $C^b ({\frak m}_K^{1-a}\Omega^{r+1}_{O_K}(\log)) \subset \Omega^{r+1}_{O_K}(\log)$. On $\Omega^{r+1}_{O_K}(\log)$, the residue map is the unique map which sends $\omega \wedge d\log(t)$ for any $\omega\in \Omega^r_{O_K}$ and any $t\in K^\times$  to $\text{ord}_K(t)\bar \omega$ where $\bar \omega$  is the image of $\omega$ in $\Omega^r_F$. 
\end{pf}

\begin{sbpara}\label{theta} Assume $[F:F^p]=p^r<\infty$. Let $a$ be an integer such that   $1 \leq a \leq e_K$. 

Fix a  ring homomorphism $\iota: F\to O_K/{\frak m}^a_K$ which lifts the identity map of $F$. 
Fix a prime element $\pi$ of $K$. Then by the isomorphism (\ref{isoTi}), we have an isomorphism
\begin{equation}
\oplus_{i=0}^{a-1} \; \Omega^r_F \overset{\cong}\to O_K/{\frak m}^a_K\otimes_{O_K} \Omega^{r+1}_{O_K}(\log)
\label{isodif}
\end{equation}
which sends $(\omega_i)_{0\leq i\leq a-1}$ to $\sum_{i=0}^{a-1} \pi^i \iota(\omega_i) \wedge d\log(\pi)$.

\end{sbpara}

\begin{sbprop}\label{Rsw3} Assume $[F:F^p]=p^r<\infty$. Let $a, b$ be  integers such that $1\leq a\leq e_K$ and $p^b \geq a$. 

1. The map 
$$R_b: {\frak m}_K^{1-a}/{\frak m}_K \otimes_{O_K} \Omega^{r+1}_{O_K}(\log)\to \Omega^r_F\;;\;\sum_{i=0}^{a-1} \pi^{-i} \otimes \iota(\omega_i) \wedge d\log(\pi)\mapsto C^b(\omega_0)$$
($\omega_i\in \Omega^r_F$), which is 
defined by  fixing $\iota$ and $\pi$ as in \ref{theta} using the isomorphism (\ref{isodif}), is independent of the choices of $\iota$ and $\pi$. 

\medskip

2. Let $\iota_b: F \to O_K/{\frak m}^a_K$ be the ring homomorphism which sends $x\in F$ to $(\tilde x)^{p^b}$ where $\tilde x$ is a lifting of $x$ to $O_K/{\frak m}^a_K$. Then for integers $i,j$ such that $i\geq 0$, $j\geq 0$ and $i+j=r+1$ and for integers $m,n$ such that $a=n-m$, the pairing $$({\frak m}_K^{-n}/{\frak m}_K^{-m}\otimes_{O_K} \Omega^j_{O_K}(\log))\times ({\frak m}_K^{m+1}/{\frak m}_K^{n+1}  \otimes_{O_K} \Omega^i_{O_K}(\log)) 
 \to \Omega^r_F\;;\; (x,y) \mapsto R_b(x\wedge y)$$
is a perfect duality of finite-dimensional  $F$-vector spaces, where $F$ acts on the two $O_K/{\frak m}^a_K$-modules on the left hand side via $\iota_b$. 

\medskip

\end{sbprop}

\begin{pf} We prove 1. If $K$ is of characteristic $p$, this follows from \ref{Res1}.  Using the isomorphism (\ref{isoTi}),  the mixed characteristic case is reduced to the positive characteristic case.

The proof of 2 is straightforward. 
\end{pf}

\subsection{Truncated exponential maps for Milnor  $K$-groups}

 Kurihara \cite{Ku} defined exponential maps of (completed) Milnor $K$-groups of complete discrete valuation fields in  mixed characteristic.
 
 Proposition \ref{Kexp} below  is a truncated version of it but works also in the positive characteristic case and outside the area of the convergence of the usual exponential map.

\begin{sbpara} Let 
$E(T)= \sum_{i=0}^{p-1} \frac{T^i}{i!}\in \Z_{(p)}[T]$.
\end{sbpara}
The following \ref{exponential} is well known. (In 2 of \ref{exponential}, note that for an integer $i\geq 1$ which is coprime to $p$, the map $1+T\Z_{(p)}[[T]]\to 1+T\Z_{(p)}[[T]]\; ;\; x \mapsto x^i$ is bijective and hence the converse map $x\mapsto x^{1/i}$ is defined on $1+T\Z_{(p)}[[T]]$.)

\begin{sblem}\label{exponential} 1. $E(T_1+T_2)\equiv E(T_1)E(T_2) \bmod (T_1,  T_2)^p\Z_{(p)}[T_1, T_2]$.

2. 
$E(T) \equiv \prod_{1\leq i\leq p-1, (i,p)=1}  (1-T^i)^{-\mu(i)/i}\bmod T^p\Z_{(p)}[[T]]$. 
Here $\mu$ is the M\"{o}bius function. 

3. $\frac{T}{E(T)}\frac{dE(T)}{dT}\equiv 1 \bmod T^p\Z_{(p)}[[T]]$.

\end{sblem}
\begin{proof} 1 (resp. 2) follows from the property $\exp(T_1+T_2)=\exp(T_1)\exp(T_2)$ in $\Q[[T_1, T_2]]$  (resp.  $\exp\left(\sum_{i\geq 0} \frac{T^{p^i}}{p^i}\right)= \prod_{i\geq 1, (i,p)=1}(1-T^i)^{-\mu(i)/i}$ in $\Q[[T]]$) of the usual exponential (of Artin-Hasse exponential) and the injectivity of $\Z_{(p)}[T_1, T_2]/(T_1,T_2)^p\to \Q[[T_1,T_2]]/(T_1,T_2)^p$ (resp. of $\Z_{(p)}[[T]]/(T^p)\to \Q[[T]]/(T^p)$). 

3 is straightforward.  
\end{proof}

\begin{sbpara} For a field $K$, let $K^M_r(K)$ be the $r$-th Milnor $K$-group of $K$. For a discrete valuation field $K$ and for $i\geq 1$, we denote by $U^iK^M_r(K)$ the subgroup of $K^M_r(K)$ generated by elements of the form $\{u_1, \dots, u_r\}$ where $u_i\in K^\times$ ($1\leq i\leq r$) and $u_1\in \text{Ker}(O_K^\times \to (O_K/{\frak m}_K^i)^\times)$.

\end{sbpara}

\begin{sbprop}\label{Kexp}  Let $K$ be a discrete valuation field and let $r\geq 0$, $t\geq 1$.

 1. We have a well-defined  homomorphism  $$E: {\frak m}_K^t/{\frak m}_K^{pt}\otimes_{O_K} \Omega^r_{O_K}(\log) \to U^tK_{r+1}^M(K)/U^{pt}K^M_{r+1}(K)$$
$$x \otimes d\log(y_1)\wedge \dots \wedge d\log(y_r)\mapsto \{E(x), y_1, \dots, y_r\}$$ 
($x\in {\frak m}_K^i$, $y_j\in K^\times$). Here $\Omega^r_{O_K}(\log)=\wedge^r_{O_K} \Omega^1_{O_K}(\log)$.

2. This map $E$ is surjective.

3. This map $E$ kills the image of 
\begin{equation}
d:  {\frak m}_K^t/{\frak m}_K^{pt}\otimes_{O_K} \Omega^{r-1}_{O_K}(\log)\to  {\frak m}_K^t/{\frak m}_K^{pt}\otimes_{O_K} \Omega^r_{O_K}(\log).
\label{eqd1}
\end{equation}
\end{sbprop}

\begin{pf} We prove 1. The $O_K$-module $\Omega^1_{O_K}(\log)$  has the following presentation by generators and relations. 

Generators. $d\log(x)$ for $x\in O_K\smallsetminus \{0\}$.

Relations. $d\log(xy)=d\log(x)+d\log(y)$ for $x,y\in O_K\smallsetminus \{0\}$. $x_0d\log(x_0)= \sum_{i=1}^n x_id\log(x_i)$ if $n\geq 1$, $x_i\in O_K\smallsetminus \{0\}$, and $x_0=\sum_{i=1}^n x_i$. 

Let $h$ be a generator of the ideal ${\frak m}_K^t$ of $O_K$.
By the above presentation of $\Omega^1_{O_K}(\log)$,  in the case $r=1$, it is sufficient to prove that $\{E(hx_0),x_0\}\equiv \sum_{i=1}^n \{E(hx_i), x_i\} \bmod U^{pt}K^M_2(K)$ if $x_i\in O_K\smallsetminus \{0\}$ and $x_0=\sum_{i=1}^n x_i$. (Here we denote the group law of $K^M_2(K)$ additively.) Since $\{E(hx_i), x_i\}=\{E(hx_i), hx_i\}-\{E(hx_i), h\}$ and since 
$\{E(hx_0),h\}\equiv \sum_{i=1}^n \{E(hx_i), h\}$ by 1 of \ref{exponential},   it is sufficient to prove that 
\begin{equation}
\{E(hx), hx\}\in U^{pt}K_2(K) \;\;\text{for}\;\; x\in O_K\smallsetminus \{0\}
\label{dkilled}
\end{equation}   
From 2 of \ref{exponential} we have, modulo $U^{pt}K^M_2(K)$, 
$$\{E(hx), hx\}\equiv  \{\prod_{\substack{1\leq i\leq p-1\\ (i,p)=1}} (1-(hx)^i)^{-\mu(i)/i}, hx\}\equiv  -\sum_{\substack{1\leq i\leq p-1\\(i,p)=1}} \mu(i)i^{-2} \{1-(hx)^i, (hx)^i\}=0.$$

The case $r\geq 2$ is reduced to the case $r=1$ and to $\{E(hx), y, y\}\equiv 0\bmod U^{pt}K_3^M(K)$  for $x\in O_K$ and $y\in K^\times$. Since $\{y,y\}=\{-1,y\}$ in $K^M_2(K)$, it is sufficient to prove that $\{E(hx), -1\}\in U^{pt}K^M_2(K)$. But this follows from the case $r=1$ because $d\log(-1)=0$ in $\Omega^1_{O_K}(\log)$. 

2 is clear. 

We prove 3. For $x\in O_K\smallsetminus \{0\}$ and $y_1, \dots, y_{r-1}\in K^\times$, $E$ sends $d(h\otimes xd\log(y_1)\wedge \dots \wedge d\log(y_{r-1}))$ to $\{E(hx), hx, y_1, \dots, y_{r-1}\}$. 
Hence 3 follows from (\ref{dkilled}). 
\end{pf}

\subsection{Cohomology of the highest degree}\label{Hd}

\begin{sbpara} Let $X$ be a Noetherian scheme of dimension $d<\infty$. Let  $\cF$ be a sheaf of abelian groups on $X$ for Zariski topology. In \ref{coh} (resp. \ref{lcoh}) below, for an abelian group $A$, we give an elementary understanding of a homomorphism $H^d(X, \cF)\to A$ (resp. $H^d_x(X, \cF)\to A$ for a closed point $x$ of $X$), where $H^d$ is the cohomology for Zariski topology
and $H^d_x$ is the cohomology with support in $x$. 

\end{sbpara}

\begin{sbpara}\label{PQX} Let $P(X)$ be the set of all $(d+1)$-tuples $\frak p=(x_0, \dots, x_d)$ of points of $X$  such that $\overline{\{x_i\}}\subsetneq  \overline{\{x_{i+1}\}}$  for all $0\leq i\leq d-1$.

For an integer $s$ such that $0\leq s\leq d$,  let $Q_s(X)$ be the set of all $d$-tuples $\frak q=(x_0, \dots, x_{s-1}, x_{s+1}, \dots, x_d)$ of points of $X$ such that the set $P_{\frak q}(X):=\{(x')_i\in P(X)\;|\; x'_i=x_i \;\text{if}\; i\neq s\}$ is not empty.
\end{sbpara}

\begin{sbpara}\label{hp} For any $\frak p=(x_i)_i\in P(X)$, we have a homomorphism $\iota_{\frak p}: \cF_{\eta(\frak p)}\to H^d(X, \cF)$ with $\eta(\frak p)=x_d$ defined as the composition 
$\cF_{\eta(\frak p)}= H^0_{x_d}(X, \cF)\to H^1_{x_{d-1}}(X, \cF)\to \dots \to H^d_{x_0}(X, \cF)\to H^d(X, \cF)$. 
(See \cite{Ha} Chap. IV for these cohomology groups with support $\{x_i\}$.) 

Hence for an abelian group $A$, a homomorphism $h: H^d(X, \cF)\to A$ induces a homomorphism $h_{\frak p}=h\circ \iota_{\frak p}: F_{\eta(\frak p)}\to A$ for each $\frak p\in P(X)$. 

\end{sbpara}

\begin{sblem}\label{coh} Let the notation be as above. Then for an abelian group $A$, the map $$\Hom(H^d(X, \cF), A) \to \prod_{\frak p\in P(X)} \Hom(\cF_{\eta(\frak p)}, A)\;;\; h\mapsto (h_{\frak p})_{\frak p\in P(X)}$$ defined above is an injection and the image consists of all elements $(h_{\frak p})_{\frak p\in P(X)}$ satisfying 
 the following conditions (i) and (ii).

\medskip
(i) Let $0\leq s\leq d-1$, $\frak q=(x_i)_i\in Q_s(X)$, and let $a\in \cF_{x_d}$.  Then $h_{\frak p}(a)=0$ for almost all $\frak p\in P_{\frak q}(X)$ and $\sum_{\frak p\in P_{\frak q}(X)}\; h_{\frak p}(a)=0$. 

(ii) Let $\frak q=(x_i)_i \in Q_d(X)$ and let $a\in \cF_{x_{d-1}}$. Then $\sum_{\frak p\in P_{\frak q}(X)} h_{\frak p}(a)=0$. (Note that $P_{\frak q}(X)$ is a finite set in this situation.)

\end{sblem}

\begin{sblem}\label{lcoh} Let the notation be as above. Let $x$ be a closed point of $X$.  Let $P_x(X)=\{\frak p=(x_i)_i\;|\; x_0=x\}$ and $Q_{s,x}(X) = \{\frak q=(x_i)_i\in Q_s(X) \;|\; x_0=x\}$ for $1\leq s\leq d$. Then for an abelian group $A$, the map
$$\Hom(H^d_x(X, \cF), A) \to \prod_{\frak p\in P_x(X)} \Hom(\cF_{\eta(\frak p)}, A)\;;\; h\mapsto (h_{\frak p})_{\frak p\in P_x(X)}$$  is an injection and the image consists of all elements $(h_{\frak p})_{\frak p\in P_x(X)}$ satisfying 
 the following conditions (i) and (ii).

\medskip
(i) Let $1\leq s\leq d-1$, $\frak q=(x_i)_i\in Q_{s,x}(X)$, and let $a\in \cF_{x_d}$.  Then $h_{\frak p}(a)=0$ for almost all $\frak p\in P_{\frak q}(X)$ and $\sum_{\frak p\in P_{\frak q}(X)}\; h_{\frak p}(a)=0$. 

(ii) Let  $\frak q=(x_i)_i \in Q_{d,x}(X)$,  and let $a\in \cF_{x_{d-1}}$. Then $\sum_{\frak p\in P_{\frak q}(X)} h_{\frak p}(a)=0$. 
\end{sblem}

\begin{sbpara}
\ref{coh} and \ref{lcoh} are proved together by induction on $\dim(X)$ as follows, by using the localization theory of cohomology with supports explained in \cite{Ha}  Chap. IV. (A version of \ref{coh} for  Nisnevich topology (not Zariski topology) is stated and proved in \cite{KSh}, Section 1.6. The proof for the Zariski topology given below is essential the same as that given there.).

Let $X_i$ be the set of all $x\in X$ such that $\dim(\overline{\{x\})})=i$. Let $X^i$ be the set of all $x\in X$ such that $\dim(\Spec(\cO_{X,x}))=i$. 
By using  the spectral sequence 
$$E_1^{ij}=\oplus_{x\in X_{-i}} H_x^{i+j}(X, \cF)\Rightarrow H^{i+j}(X, \cF)$$
and using $H^i_x(X, \cF)= H^i_x(\Spec(\cO_{X,x}), \cF)$, where we denote the pullback of $\cF$ to $\Spec(\cO_{X,x})$ also by $\cF$, we obtain the following 1--5 by induction on $\dim(X)$.

1. If $i>d$, $H^i(X, \cF)=0$ and $H^i_x(X, \cF)=0$ for any  $x\in X$. 

2. We have an exact sequence $\oplus_{x\in X_1} H^{d-1}_x(X, \cF) \to \oplus_{x\in X_0} H^d_x(X, \cF)\to H^d(X, \cF)\to 0$.

3. If $x\in X^i$ and $i\geq 2$, $H^i_x(X, \cF)\cong H^{i-1}(\Spec(\cO_{X,x})\smallsetminus \{x\}, \cF)$ and $\Spec(\cO_{X,x})\smallsetminus \{x\}$ is of dimension $i-1$.

4. If $x\in X^1$, we have an exact sequence $\cF_x\to \oplus_{\eta\in \Spec(\cO_{X,x})\smallsetminus \{x\}}\cF_{\eta}\to H^1_x(X, \cF)\to 0$.

5. If $x\in X^0$, $H^0_x(X, \cF)= \cF_x$. 

\ref{coh} and \ref{lcoh} follow from these 1--5 by induction on $d$.

\end{sbpara}

\subsection{Higher dimensional class field theory (review)}

We review higher dimensional local class field theory in \cite{Pa}, \cite{KK0}, etc. in \ref{dlocal}--\ref{loc10} and higher dimensional global class field theory in \cite{KSh}, etc. in \ref{gl1}--\ref{gl10} briefly, giving  complements \ref{cpl1}--\ref{cpl10} to the  relation between local theory and global theory.

We first review the higher dimensional local class field theory.

\begin{sbpara}\label{dlocal} Recall that the notion $d$-dimensional local field is as follows. A $0$-dimensional local field is a finite field. For $d\geq 1$, a $d$-dimensional local field is a complete discrete valuation field whose residue field is a $d-1$-dimensional local field. 

\end{sbpara}

\begin{sbpara}\label{VK} The following remark is used in \ref{cpl9} later.

To give a $d$-dimensional local field $K$ is equivalent to giving a valuation ring $V$ having the following properties (i)--(iii). (i) The residue field of $V$ is a finite field. (ii) The value group of $V$ is isomorphic to $\Z^d$ with the lexicographic order. (iii) If $P_0\supsetneq \dots \supsetneq P_d=(0)$ are all prime ideals of $V$, for each $1\leq i\leq d$, the local ring of $V/P_i$ at the prime ideal $P_{i-1}/P_i$ (which is a discrete valuation ring) is complete. 

In fact, $K$ is obtained from $V$ as the field of fractions of $V$. $V=V_K$ is a subring of $K$ defined by induction on $d$ as follows. If $d=0$, then $V_K=K$. If $d\geq 1$, $V_K$ is the subring of the discrete valuation ring $O_K$ consisting of all elements whose images in the residue field $F$ of $K$ belong to $V_F$.

\end{sbpara}

\begin{sbpara}\label{loc1}

By higher dimensional local field theory (see \cite{Pa}, \cite{KK0}), for a $d$-dimensional local field $K$, we have a canonical homomorphism $K^M_d(K)\to \Gal(K^{\ab}/K)$ called the reciprocity map, where $K^{\ab}$ is the maximal abelian extension of $K$.

\end{sbpara}

\begin{sbpara}\label{loc2}  For $n\geq 0$, define the category $\cF_n$ inductively as follows. $\cF_0$ is the category of finite sets. For $n\geq 1$, $\cF_n= \text{ind}(\text{pro}(\cF_{n-1}))$ where $\text{pro}(\;)$ is the category of pro-objects and $\text{ind}(\;)$ is the category of ind-objects. Let $\cF_{\infty}= \cup_n \cF_n$. 

\end{sbpara}

\begin{sbpara}\label{loc10} 
For a $d$-dimensional local field $K$, we can regard $K$ and $K^\times$ as objects of $\cF_{\infty}$ (actually of $\cF_d$) canonically (\cite{KK3}). A homomorphism $K^M_d(K)\to \Q/\Z$ is said to be continuous if the composition $K^\times \times \dots \times K^\times\; (\text{$d$ times}) \overset{\{\;\}}\to K^M_d(K)\to \Q/\Z$ is a morphism of $\cF_{\infty}$.

The main result of higher dimensional local class field theory formulated in the style of \cite{KK3} is the following. Via the reciprocity map in \ref{dlocal}, $H^1(K, \Q/\Z)= \Hom_{\text{cont}}(\Gal(K^{\ab}/K), \Q/\Z)$ is isomorphic to the group of continuous homomorphisms $K^M_d(K) \to \Q/\Z$ of finite orders.   
\end{sbpara}

We now consider the global theory. In the higher dimensional global class field theory in \cite{KSh}, ''henselian variants of higher dimensional local fields'' (the fields $K^h_v$ below) are used. But we use in this paper the higher dimensional local fields $K_v$ (see below). We give complements to relate $K_v$ to the global class field theory in \cite{KSh}.

\begin{sbpara}\label{cpl1} Let $X$ be an integral scheme of finite type over $\Z$ of dimension $d$, and let $K$ be the function field of $X$. By a  place of $K$ along $X$, we mean a subring $v$ of $K$ satisfying the following conditions (i)--(iii).

(i) $K$ is the field of fractions of $v$.

(ii) $v$ is a valuation ring and the value group is isomorphic to $\Z^d$ with the lexicographic order.

(iii) There is an element $\frak p=(x_i)_{0\leq i \leq d}\in P(X)$ such that  if  $P_0\supsetneq \dots \supsetneq P_d$ denotes the set of all prime ideals of $v$, then for $0\leq i\leq d$, the local ring of $v$ at $P_i$ (which is a valuation ring) dominates  $x_i$.

Let $Pl(X)$ be the set of all places of $K$ along $X$. 
We have a canonical map $Pl(X) \to P(X)$ which sends $v\in Pl(X)$ to $\frak p$ in (iii). 

For $v\in Pl(X)$, we will define a $d$-dimensional local field $K_v\supset K$ called the local field of $K$ at $v$. The class field theory of $K$ is related to the local class field theory of $K_v$ as is explained below. 

\end{sbpara}

\begin{sbpara}\label{cpl2} To obtain local fields $K_v$ of $K$  for $v\in Pl(X)$ and their henselian versions $K^h_v$, we use the following iterated completion and iterated henselization.

For $0\leq i\leq d$ and for a sequence $(x_0, \dots, x_i)$ of points of $X$ such that $\overline{\{x_0\}}\subset \dots\subset  \overline{\{x_i\}}$ and such that $\dim(\overline{\{x_j\}})=j$ for $0\leq j\leq i$, we define the rings ${\hat \cO}_{X, x_0, \dots, x_i}$ and  $\cO^h_{X, x_0, \dots, x_i}$ over $\cO_{X,x_i}$ as follows, inductively.

In the case $i=0$,  ${\hat \cO}_{X, x_0}$ (resp. $\cO^h_{X, x_0}$) is the completion (resp. henselization) of $\cO_{X, x_0}$.

For $i\geq 1$, let $${\hat \cO}_{X,x_0, \dots,x_i}= \prod_P  ({\hat \cO}_{X, x_0, \dots, x_{i-1}})_P^{\wedge}, \quad \cO^h_{X,x_0, \dots,x_i}= \prod_P  (\cO^h_{X, x_0, \dots, x_{i-1}})_P^h$$
where $P$ ranges over 
  all prime ideals of ${\hat \cO}_{X, x_0, \dots, x_{i-1}}$ (resp. $\cO^h_{X, x_0, \dots, x_{i-1}}$) lying over the prime ideal of $\cO_{X,x_{i-1}}$ corresponding to $x_i$, and $(\;)_P^{\wedge}$ (resp. $(\;)_P^h$) denotes the completion (resp. henselization) of the local ring $(\;)_P$ of the ring $(\;)$ at $P$.  By induction on $i$, we see that the set of such $P$ is a non-empty finite set and that  ${\hat \cO}_{X, x_0, \dots, x_i}$ and  $\cO^h_{X, x_0, \dots, x_i}$ are finite products of complete (resp. henselian) local integral domains of dimension $d-i$.  
  
For $\frak p=(x_i)_i \in P(X)$, let $$K_{\frak p}= {\hat \cO}_{X,x_0, \dots, x_d}, \quad K^h_{\frak p}= \cO^h_{X, x_0, \dots, x_d}.$$ Then  $K_{\frak p}$ and $K^h_{\frak p}$ are finite products of fields.

\end{sbpara}

\begin{sbpara}\label{cpl9} Let $X$ and $K$ be as above. The following statements are proved in \ref{cpl10} by induction on $d$.

1. Let $\frak p\in P(X)$, and let $Pl(\frak p)=Pl(X, \frak p)$ be the inverse image of 
 $\frak p$ in $Pl(X)$. Then each field factor of $K_{\frak p}$ has canonically a structure of a $d$-dimensional local field. We have a bijection $\Spec(K_{\frak p})\to Pl(\frak p)$ which sends the point of $\Spec(K_{\frak p})$ corresponding to a field factor $L$ of $K_{\frak p}$ to the valuation ring $V_L\cap K$, where $V_L\subset L$ is the valuation ring of rank $d$ associated to $L$ by \ref{VK}. In particular, the set $Pl(\frak p)$ is a non-empty finite set. 

2. Let $\frak p\in P(X)$. Then the map $\Spec(K_{\frak p})\to \Spec(K^h_{\frak p})$ induced by the inclusion map $K^h_{\frak p} \to K_{\frak p}$ is bijective. 

3.  Let $X' \to X$ be a finite surjective morphism of integral schemes of finite type over $\Z$ and let $K'$ be the function field of $X'$. Then we have  canonical isomorphisms
$$K'\otimes_K  K_v= \prod_{v'|v} K'_{v'}, \quad K'\otimes_K K^h_v= \prod_{v'} (K')^h_{v'}$$
where $v'|v$ means that $v'$ ranges over all elements of $Pl(X')$ such that $v'\cap K=v$. 

\end{sbpara}

\begin{sbpara}\label{cpl10} We prove statements in \ref{cpl9} by induction on $d$. 

We prove 1. Let $\frak p=(x_i)_i\in P(X)$, and let $Y\subset X$ be the closure of $x_{d-1}$ in $X$ with the reduced scheme structure. Let ${\frak q}=(x_0,\dots, x_{d-1})\in P(Y)$.
Let $F$ be the residue field of $x_{d-1}$, that is, the function field of $Y$. 
Let $A$ be the normalization of $\cO_{X, x_{d-1}}$ and let 
 $\Delta$ be the set of all maximal ideals of $A$. Then $\Delta$ is finite. For $z\in \Delta$, let $\kappa(z)$ be the residue field of $z$ which is a finite extension of $F$,  let $Y(z)$ be the integral closure of $Y$ in $\kappa(z)$ and let $Pl(Y(z), \frak q)\subset Pl(Y(z))$ be the inverse image of $\frak q$ under the map  $Pl(Y(z))\to P(Y)$. Then we have a bijection 
$\coprod_{z\in \Delta}  Pl(Y(z), \frak q)\to Pl(\frak p)$ which sends  $w\in Pl(Y(z), \frak q)$ with $z\in \Delta$ to the valuation ring consisting of all elements of the local ring of $A$ at $z$ whose residue classes in $\kappa(z)$ belong to $w$.  
 The ring $A$ is Noetherian normal one-dimensional semi-local integral domain, and hence is a PID. Let $t$ be a generator of the intersection of all maximal ideals of $A$. Let $B={\hat \cO}_{X,x_0,\dots, x_{d-1}}$ and let $C=A\otimes_{\cO_{X,x_{d-1}}} B$. Then $C$ is a finite product of one-dimensional local integral domains,   $C/tC=(\prod_{z\in \Delta} \kappa(z)) \otimes_F  B/{\frak m}B$ where $\frak m$ is the maximal ideal of $\cO_{X, x_{d-1}}$, and $B/{\frak m}B= {\hat \cO}_{Y, x_0, \dots, x_{d-1}}=\prod_{\lambda\in Pl(\frak q)} F_{\lambda}$ where the last $=$ is by the statement 1 for dimension $d-1$. 
 Hence 
$\kappa(z) \otimes_F B/{\frak m}B= \kappa(z) \otimes_F \prod_{\lambda\in Pl(\frak q)} F_{\lambda}= \prod_{w\in Pl(Y(z), \frak q)} \kappa(z)_w$ by the statement 3 for dimension $d-1$.  This shows that $\Spec(K_{\frak p})$ is identified with $\coprod_{z\in \Delta} Pl(Y(z), \frak q)$. Hence $\Spec(K_{\frak p})$ is identified with $Pl(\frak p)$, and for $v\in Pl(\frak p)$ corresponding to $w\in Pl(Y(z), \frak q)$, $K_v$  is the field of fractions of a complete discrete valuation ring whose residue field is $\kappa(z)_w$. This proves 1. 

We prove 2. Replacing the iterated completion in the above arguments by the iterated henselization, and by using the induction on $d$,  we obtain similarly a bijection between $\Spec(K^h_{\frak p})$ and $Pl(\frak p)$. This shows that the map $\Spec(K_{\frak p})\to \Spec(K^h_{\frak p})$ is bijective.

We prove 3 for $K_v$. Let $f$ be the morphism $X'\to X$. By induction on $i$, we have that 
$(f_*\cO_{X'})_{x_i} \otimes_{\cO_{X,x_{d-1}}} {\hat \cO}_{X,x_0,\dots, x_i}\overset{\cong}\to  \prod_{x'_0, \dots, x'_i}  {\hat \cO}_{X', x_0',\dots, x'_i}$ where $(x'_0, \dots, x'_i)$ ranges over all sequences of points of $X'$ such that $x'_j$ lies over $x_j$ for $0\leq j\leq i$ and such that $\overline{\{x'_0\}}\subset \dots \subset \overline{\{x'_i\}}$. 
The case $i=d$ gives an isomorphism $K'\otimes_K K_{\frak p}\overset{\cong}\to \prod_{\frak p'} K'_{\frak p'}$ where $\frak p'$ ranges over all elements of $P(X')$ lying over $\frak p$. By looking at the $v$-factor of this isomorphism, we have 3 for $K_v$.

The proof of 3 for $K^h_v$ is similar to that for $K_v$.

\end{sbpara}

\begin{sbpara}\label{gl1}  In \ref{gl1}--\ref{gl10}, let $X$ be a proper normal integral scheme over $\Z$ and let $K$ be the function field of $X$. We assume $X(\R)=\emptyset$.

The following is what we use in this paper from the higher dimensional global class field theory.

We have a unique continuous homomorphism 
\begin{equation}
\varprojlim_I H^d(X, K^M_d(\cO_X, I)) \to \Gal(K^{ab}/K)
\label{glrec1}
\end{equation}
called the reciprocity map, where $I$ ranges over all non-zero coherent ideals of $\cO_X$, $K^M_d(\cO_X, I)$ denotes the kernel of $K^M_d(\cO_X)\to K^M_d(\cO_X/I)$ with $K_d^M$ the sheaf of $d$-th Milnor $K$-groups, the cohomology groups are the  Zariski cohomology groups, and the left hand side is regarded a topological space for the projective limit of the discrete topologies of $H^d(X, K^M_d(\cO_X,I))$, which is characterized by the following property. For $\chi\in H^1(K, \Q/\Z)= \Hom_{\text{cont}}(\Gal(K^{\ab}/K), \Q/\Z)$, the homomorphism $H^d(X, K^M_d(\cO_X, I))\to \Q/\Z$ induced by $\chi$ for some $I$ corresponds to $(h_{\frak p})_{\frak p\in P(X)}$ via \ref{coh} where $h_{\frak p}: K^M_d(K)\to \Q/\Z$ is the composition 
$$K^M_d(K) \to \oplus_{v\in Pl(\frak p)} K_d^M(K_v)\overset{\chi}\to \Q/\Z$$
where $Pl(\frak p)$ is as in \ref{cpl9}, the first arrow is the diagonal map, and the map $K^M_d(K_v)\to \Q/\Z$ is induced by the image of $\chi$ in $H^1(K_v, \Q/\Z)$ and by the $d$-dimensional local class field theory of $K_v$. 
\end{sbpara}

\begin{sbpara}\label{gl2}  In \cite{KSh}, a continuous map  
\begin{equation}
\varprojlim_I H^d(X_{\text{Nis}}, K^M_d(\cO_X, I)) \to \Gal(K^{ab}/K)
\label{glrec2}
\end{equation}
called the reciprocity map is defined, where the cohomology groups are  Nisnevich cohomology,  instead of Zariski cohomology,  and the left hand side is endowed with the topology of the projective limit of the discrete sets.

The map (\ref{glrec1}) is induced from the map (\ref{glrec2}) and the canonical map $H^d(X, K_d^M(\cO_X, I)) \to H^d(X_{\text{Nis}}, K^M_d(\cO_X,I))$. 

In fact, by the definition of the reciprocity map (\ref{glrec2}) in \cite{KSh}, for $\chi\in H^1(K, \Q/\Z)$, the induced homomorphism (\ref{glrec1}) corresponds to $(h_{\frak p})_{\frak p\in P(X)}$ where $h_{\frak p}: K^M_d(K)\to \Q/\Z$ is the composition $K_d^M(K) \to \oplus_{v\in Pl(\frak p)} K^M_d(K^h_v)\overset{\chi}\to \Q/\Z$ with  $K^M_d(K^h_v)\to \Q/\Z$ a homomorphism defined in \cite{KSh}. This last homomorphism  coincides with the composition $K^M_d(K^h_v) \to K^M_d(K_v) \to \Q/\Z$ where the last map is that in \ref{gl1}. 

\end{sbpara}

\begin{sbpara}\label{gl10}
The main result of \cite{KSh} concerning  the class field theory of $K$ is as follows. 
(In \cite{KSh}, we do not need to assume $X(\R)=\emptyset$, but if we do not assuming $X(\R)=\emptyset$, the left hand side of (\ref{glrec2}) should be modified by adding archimedean objects to have results below).

In the case $K$ is of characteristic $0$, the map (\ref{glrec2}) is an isomorphism of topological groups.  

 In the case $K$ is of characteristic $p>0$, the map (\ref{glrec2}) induces an isomorphism of topological groups from the left hand side of (\ref{glrec2}) to the fiber product of $\Gal(K^{\ab}/K) \to \Gal(\F_p^{\ab}/\F_p) \leftarrow \Z$, where $\Z$ is discrete and the right arrow sends $1\in \Z$ to the Frobenius $\F_p^{\ab}\to \F_p^{\ab}\;;\;x\mapsto x^p$. 

By Raskind \cite{Ra} and Kerz and Saito (\cite{KeS}), these results on  the  class field theory of $K$ hold also for Zariski cohomology $\varprojlim_I H^d(X, K^m_d(\cO_X, I))$ (replacing Nisnevich cohomology)  if $K$ is of characteristic $\neq  2$.

\end{sbpara}

\begin{sbpara} There is another formulation of higher dimensional class  field theory due to Wiesend (\cite{Wi}) which was studied more in \cite{KeS}. But we do not use it in this paper. 

\end{sbpara}

\begin{sbrem}\label{gl4}

The first author would like to take this opportunity to express that the both authors of \cite{KSh} regret that Nisnevich topology is called henselian topology in \cite{KSh} due to their ignorance of the preceding works of Nisnevich. 

\end{sbrem}

\section{Refined Swan conductors mod $p$}

\subsection{The subject}

\begin{sbpara}\label{Rsw1} Let $K$ be a complete discrete valuation field with residue field $F$, and assume $F$ is of characteristic $p>0$. Let $n\geq 1$ and let $m= \max(n-e_K, [n/p])$. 

We  define a homomorphism $$\Rsw: F_nH^1(K, \Q/\Z) \to {\frak m}_K^{-n}/{\frak m}_K^{-m} \otimes_{O_K} \Omega^1_{O_K}(\log)$$
which we call the refined Swan conductor modulo $p$.

\end{sbpara}
\begin{sbpara}\label{Rsw4} The homomorphism $\Rsw$ in \ref{Rsw1} is characterized by the following properties (i) and (ii). 

\medskip

(i)  $\Rsw$ is compatible with any homomorphisms of cdvf. That is,  the following diagram is commutative for an extension of complete discrete valuation fields $K'/K$, where $n'=e(K'/K)n$ with $e(K'/K)$ the ramification index of the extension $K'/K$ and $m'=\max(n'-e_{K'}, [n'/p])$. 

$$\begin{matrix}  
F_nH^1(K, \Q/\Z)& \to & {\frak m}_K^{-n}/{\frak m}_K^{-m} \otimes_{O_K} \Omega^1_{O_K}(\log)\\
\downarrow & & \downarrow \\
F_{n'}H^1(K', \Q/\Z)& \to & {\frak m}_{K'}^{-n'}/{\frak m}_{K'}^{-m'} \otimes_{O_{K'}} \Omega^1_{O_{K'}}(\log)
\end{matrix}$$
(The fact $F_nH_1(K, \Q_p/\Z_p)$ is sent to $F_{n'}H^1(K', \Q_p/\Z_p)$ is proved in \cite{KKs}.) 

(ii) If $F$ is an $r$-dimensional local field ($r\geq 0$), Rsw is characterized by the  property
$$\chi(E(\alpha))=Res_F(R_b(\alpha\wedge \Rsw(\chi)))$$
for $\chi\in F_nH^1(K, \Q_p/\Z_p)$, $\alpha\in {\frak m}_K^{m+1}/{\frak m}_K^{n+1}\otimes_{O_K} \Omega^1_{O_K}(\log)$. Here $E(\alpha)$ denotes the image of $\alpha$ under the truncated exponential map $${\frak m}_K^{m+1}/{\frak m}_K^{n+1}\otimes_{O_K} \Omega^1_{O_K}(\log)\to 
 U^{m+1}K^M_{r+1}(K)/U^{n+1}K^M_{r+1}(K),$$ 
 $\chi(E(\alpha))$ denotes the image of $E(\alpha)$ under 
the composition $K^M_{r+1}(K)\to \Gal(K^{\ab}/K)\overset{\chi}\to \Q_p/\Z_p$ of the reciprocity map of the $r+1$-dimensional local field $K$ and $\chi$, which kills $U^{n+1}K^M_{r+1}(K)$ by the condition $\text{Sw}(\chi)\leq n$, 
and  $$\Res_F:\Omega^r_F \to \F_p$$ is the residue map (\ref{rdl}). 
Note that $\alpha\wedge \Rsw(\chi)\in {\frak m}_K^{m+1-n}/{\frak m}_K \otimes_{O_K} \Omega^{r+1}_{O_K}(\log)$ and $n-m\leq e_K$ and hence its $R_b$ is defined in $\Omega^r_F$. 

\end{sbpara}

\begin{sbpara} If $\text{char}(K)=p>0$, 
$$\Rsw: F_nH^1(K, \Q/\Z) \to {\frak m}_K^{-n}/{\frak m}_K^{-[n/p]} \otimes_{O_K} \Omega^1_{O_K}(\log)$$
is the homomorphism of
 Matsuda  \cite{Ma} Rem. 3.2.2 and  Borger \cite{Bo2}  3.6 defined by using Artin-Schreier-Witt theory. It was studied in \cite{IL} Section 2.
See Section \ref{pos} and \ref{charp=}.

In the mixed characteristic case, we will construct $\Rsw$ using higher dimensional local class field theory (Section 3.3) and also higher dimensional global class field theory (Section 3.4). 

\end{sbpara}

\begin{sbpara} It may seem strange that we use higher dimensional global class field theory for the local subject $\Rsw$. Our idea is that since any field is a union of finitely generated fields over a prime field and since each finitely generated field over a prime field has  class field theory, any subject about one-dimensional Galois representations of any field can be studied by using higher dimensional global class field theory. Our method in \ref{union} to define $\Rsw$ actually follows this idea.  

A purely local method for the local subject $\Rsw$ exists also in the mixed characteristic case as in \ref{localmethod} below. 

\end{sbpara}

\begin{sbpara}\label{localmethod} In the case $F$ is perfect, our homomorphism Rsw can be defined also by using the local class field theory of Serre \cite{Se} and Hazewinkel \cite{HM}.

In the case $[F:F^p]=p^r<\infty$, Rsw can be also defined by using the duality theory \cite{KS} (a work in preparation) which is  a generalization of the local class field theory of Serre \cite{Se} and Hazewinkel \cite{HM} to the case $[F:F^p]=p^r$.  

These things will be explained in \cite{KS}. 

\end{sbpara}

\begin{sbpara}

The authors wonder whether Rsw in general can be obtained by the reduction to the perfect residue field case using the local class field theory of Serre \cite{Se} and Hazewinkel \cite{HM} and 
 the work \cite{Bo1} of Borger.
\end{sbpara}

\begin{sbpara}
The authors  wonder whether our refined Swan conductor mod $p$ can be obtained from the relation of $p$-adic \'etale cohomology and Hochschild homology in \cite{BMS}.
\end{sbpara}

\subsection{Positive characteristic case (review)}\label{pos}

\begin{sbpara} Let $K$ be a complete discrete valuation field of characteristic $p>0$. We briefly review the definition of 
$$\Rsw: F_nH^1(K, \Q/\Z) \to {\frak m}_K^{-n}/{\frak m}_K^{-[n/p]} \otimes_{O_K} \Omega^1_{O_K}(\log).$$ For details, see \cite{Bo2,Ma,IL}. 

In \ref{charp=}, we will show that this Rsw has the properties (i) and (ii) in \ref{Rsw4}. 

\end{sbpara}

\begin{sbpara}
From Artin-Schreier-Witt theory, there are isomorphisms 
\begin{equation}
W_s(K)/(\phi-1)W_s(K)\simeq H^1(K,\Z/p^s \Z),
\label{ASW}
\end{equation}
where $W_s(K)$ denotes the ring of Witt vectors of length $s$, and $\phi$ the endomorphism of Frobenius.

\end{sbpara}

\begin{sbpara}

For a Witt vector $a=(a_{s-1},\ldots, a_0)\in W_s(K)$, let $\text{ord}_K(a)=\min\limits_i\{p^i \text{ord}_K(a_i)\}$.
In \cite{Br}, Brylinski defined an increasing filtration of $W_s(K)$ as 
$$
F_n W_s(K) = \{a \in W_s(K)\;|\; \text{ord}_K(a) \geq -n\}
$$
for $n\in \Z_{\geq0}$. The filtration $F_n H^1(K,\Z/p^s \Z)$  defined in \cite{KKs} is the image of 
 $F_n W_s(K)$ under  (\ref{ASW}).
 
 \end{sbpara}
 
 \begin{sbpara}

 The homomorphism $-d:W_s(K)\to \Omega_{K}^1$ given by $a=(a_{s-1},\ldots, a_0)\mapsto  -\sum\limits_i a_i^{p^i-1}da_i$  produces a map 
 $$
 F_n W_s(K)/F_{[n/p]}W_s(K)\to {\frak m}_K^{-n}/{\frak m}_K^{-[n/p]} \otimes_{O_K} \Omega^1_{O_K}(\log), 
 $$
  which factors through 
  $$F_nH^1(K, \Z/p^s\Z)/F_{[n/p]}H^1(K,\Z/p^s\Z) \to {\frak m}_K^{-n}/{\frak m}_K^{-[n/p]} \otimes_{O_K} \Omega^1_{O_K}(\log),$$
  inducing
  $$\Rsw: F_nH^1(K, \Q/\Z) \to {\frak m}_K^{-n}/{\frak m}_K^{-[n/p]} \otimes_{O_K} \Omega^1_{O_K}(\log).$$
  
  Here the minus sign of the definition of $-d$ may seem strange, but we put it to have the compatibility with the refined Swan conductor mod ${\frak m}_K$ in \cite{KKs}. The minus sign naturally appears in the argument in \ref{pfcharp=}. 
  
  \end{sbpara}
  
  \begin{sbrem}
	  The second author was unaware of the unpublished work of Borger \cite{Bo2} when writing \cite{IL}, and sincerely regrets not quoting it in \cite{IL}.
  \end{sbrem}

\subsection{Application of higher dimensional local class field theory}\label{local}

\begin{sbpara} Let $K$ be a complete discrete valuation field with residue field $F$.  In the case $F$ is an $r$-dimensional local field, we define Rsw by using  higher dimensional local class field theory as below. 

We will use the continuity of $K^M_{r+1}(K)\to \Q/\Z$ induced by the reciprocity map of the $r+1$-dimensional local field $K$ and by $\chi\in H^1(K, \Q/\Z)$ and we will use also the self-duality of the additive group $F$. Here the continuity is the one defined in \cite{KK3}, and  the duality is also treated by using such continuity. (In the case $r=1$, this duality is the usual self-duality of the locally compact abelian group $F$.)

\end{sbpara}

\begin{sbpara}\label{rdl}  Let $F$ be an $r$-dimensional local field of characteristic $p>0$. 
We have the residue map $$\Res_F: \Omega_F^r\to \F_p$$
defined as follows. (The following definition and properties of $\Res_F$ are contained in \cite{KK1} Chap. 2, Lemma 12 and Lemma 14.) Let $F_r=F$ and for $1\leq i\leq r$, define the field $F_{i-1}$ to be the residue field of $F_i$ by downward induction on $i$. Then $\Res_F$ is the composition 
$$\Omega^r_F =\Omega^r_{F_r} \overset{\Res}\to \Omega^{r-1}_{F_{r-1}}\to \dots \to \Omega^1_{F_1} \overset{\Res}\to \Omega^0_{F_0}=F_0\overset{\text{trace}}\to \F_p$$
where each 
 $\Res: \Omega^i_{F_i} \to \Omega^{i-1}_{F_{i-1}}$ is the residue map in \ref{Res1} defined by taking $F_{i-1}\to O_{F_i}$ and a prime element $\pi_i$ of the discrete valuation field $F_i$ as in \ref{Res1}. This composition $\Res_F$  is independent  of the choices of $F_{i-1}\to O_{F_i}$ and $\pi_i$. 
 
 We have 
 \begin{equation}
  \Res_F\circ C=\Res_F,
  \label{ResC}
  \end{equation}
   where $C$ is the Cartier operator. For a  finite extension $F'$ of $F$, we have
 \begin{equation}
 \Res_F\circ Tr_{F'/F}= \Res_{F'},
 \label{ResTr}
 \end{equation}
   where $Tr_{F'/F}$ is the trace map $\Omega^r_{F'}\to \Omega^r_F$. (This trace map for differential forms is defined in \cite{KK1} as $$\Omega^r_{F'}\cong U^1K^M_{r+1}(F'((T)))/U^2K^M_{r+1}(F'((T))) $$ $$\overset{\text{norm}}\to U^1K^M_{r+1}(F((T)))/U^2K^M_{r+1}(F((T)))\cong \Omega^r_F.$$ The trace map for differential forms in a more general setting  is defined in \cite{Ga}.  The formula (\ref{ResTr}) contains the formula (\ref{ResC}) because
  for a field $k$ of characteristic $p$ such that $[k:k^p]=p^r$, the Cartier operator $\Omega^r_k \to \Omega^r_k$ is the trace map of the homomorphism $k\to k\; ;\;x\mapsto x^p$.) 
   
 \end{sbpara}

\begin{sbpara}\label{rdl2}  Let $F$ be as in \ref{rdl} and let $V$ be a finite dimensional vector space over $F$. Then $V$ is canonically regarded as an object of the category $\cF_{r,\ab}$ of abelian group objects of the category $\cF_r$ (\ref{loc2}). 

\end{sbpara}

\begin{sblem}\label{rdl4} Let $F$ and $V$ be as in \ref{rdl2} and let $V^*=\Hom_F(V, \Omega^r_F)$. Then we have the bijection
$$V^*\overset{\cong}\to \Hom_{\text{cont}}(V, \F_p)\;;\;h\mapsto (x\mapsto \Res(h(x))).$$
Here $\Hom_{\text{cont}}$ is the set of homomorphisms in the category  $\cF_{r,\ab}$. 

\end{sblem}

\begin{pf} This is shown by induction on $r$. \end{pf}

\begin{sbpara}  Let $K$ be a complete discrete valuation field with residue field $F$ and assume that $F$ is an $r$-dimensional local field of characteristic $p>0$. (Hence $K$ is an $r+1$-dimensional local field.)

Let $\chi\in F_nH^1(K, \Q_p/\Z_p)$. 
Then by \cite{KS}, $\chi$ induces a continuous homomorphism $K^M_{r+1}(K)/U^{n+1}K^M_{r+1}(K)\to \Q_p/\Z_p$, where the continuity is that of \cite{KK3}. 
Let $m=\max(n-e_K,[n/p])$. Then via $E: {\frak m}_K^{m+1}/{\frak m}_K^{n+1}\otimes_{O_L} \Omega^r_{O_K}(\log) \to K^M_{r+1}(K)/U^{n+1}K^M_{r+1}(K)$, we obtain a continuous homomorphism 
$u_{\chi}: {\frak m}_K^{m+1}/{\frak m}_K^{n+1}\otimes_{O_K} \Omega^r_{O_K}(\log) \to \F_p$. 

Take an integer $b\geq 0$ such that $p^b\geq n-m$. Then 
$$({\frak m}_K^{-n}/{\frak m}_K^{-m}\otimes_{O_K} \Omega^1_{O_K}(\log))\times ({\frak m}_K^{m+1}/{\frak m}_K^{n+1}\otimes_{O_K} \Omega^r_{O_K}(\log)) \to \Omega^r_F\;;\;(x,y)\mapsto R_b(x\wedge y)$$
is a perfect duality of finite-dimensional $F$-vector spaces where $F$ acts on the two $O_K/{\frak m}^{n-m}_K$-modules on the left hand side via $\iota_b$ (2 of \ref{Rsw3}). Hence by \ref{rdl4}, $u_{\chi}$ corresponds to an element $\Rsw_{K_v}(\chi)\in {\frak m}_K^{-n}/{\frak m}_K^{-m}\otimes_{O_K} \Omega^1_{O_K}(\log)$. This element is independent of the choice of $b$. 

\end{sbpara}

The following lemma will be used in Section \ref{s3.5}.

\begin{sblem}\label{lem3} Assume that $K$ is a $d$-dimensional local field whose residue field $F$ is of characteristic $p>0$. Assume one of the following (i) and (ii). 

(i) $K'$ is a finite extension of $K$.

(ii) $K'$ is the field of fractions $K\{\{T\}\}$ of the completion of the local ring of $O_K[[T]]$ at the prime ideal generated by ${\frak m}_K$ (then $K'$ is a $d+1$-dimensional local field with residue field $F((T))$).

Then Rsw$_K$ and Rsw$_{K'}$   are compatible. 

\end{sblem}

\begin{sbpara}\label{pflem3i} We prove the case (i) of \ref{lem3}. 
We have a  commutative diagram
$$\begin{matrix}  K_d^M(K')  &\to & \Gal((K')^{\ab}/K')\\
\downarrow && \downarrow\\
K_d^M(K) &\to& \Gal(K^{\ab}/K)
\end{matrix}$$
where the left vertical arrow is the norm map and the right vertical arrow is the natural one (\cite{KK0}, Section 3.2, Cor. 1). 

Let $R_K: {\frak m}^{1-n+m}_K/{\frak m}_K \otimes_{O_K} \Omega^d_{O_K}(\log)\to \F_p$ be the composite map $\Res_F\circ R_b$ ($p^b\geq n-m$) which is independent of $b$, and let $R_{K'}:  {\frak m}_{K'}^{1-n'+m'}/{\frak m}_{K'} \otimes_{O_{K'}} \Omega^d_{O_{K'}}(\log)\to \F_p$  ($n'=e(K'/K)n$, $m'= \max(n'-e_{K'}, [n'/p])$) be the corresponding map for $K'$. 

Let $\chi\in F_nH^1(K, \Q_p/\Z_p)$, let $\chi_{K'}\in F_{n'}H^1(K', \Q_p/\Z_p)$ be the image of $\chi$, and let $\Rsw(\chi)_{K'}\in {\frak m}_{K'}^{-n'}/{\frak m}_{K'}^{-m'} \otimes_{O_{K'}} \Omega^1_{O_{K'}}(\log)$
be the image of $\Rsw(\chi)\in {\frak m}_K^{-n}/{\frak m}_K^{-m}\otimes_{O_K} \Omega^1_{O_K}(\log)$.  Since the Galois group of any finite Galois extension of a higher dimensional local field is solvable, any finite extension of a higher dimensional local field is a successive extension of extensions whose degrees are prime numbers. Hence it is sufficient to prove $\Rsw(\chi_{K'})=\Rsw(\chi)_{K'}$ in the case  $[K':K]$ is a prime number. It is sufficient to prove in this case that
\begin{equation}
R_{K'}(\Rsw(\chi_{K'})\wedge \alpha)= R_{K'}(\Rsw(\chi)_{K'}\wedge \alpha)
\label{R=R}
\end{equation}
for any $\alpha\in {\frak m}^{m'+1}_{K'}/{\frak m}^{n'+1}_{K'}\otimes_{O_{K'}} \Omega^{d-1}_{O_{K'}}(\log)$. In this case, the last group is generated additively by elements of the following three forms. 

(a)  $f \omega$ where $f\in {\frak m}_{K'}^{m'+1}$ and $\omega\in  \Omega^{d-1}_{O_K}(\log)$. 

(b)  $\omega \wedge df$ where $f$ is as in (a) and $\omega\in \Omega_{O_K}^{d-2}(\log)$. 

(c) $\omega\wedge d\log(f)$ where $\omega\in {\frak m}^{m+1}_K/{\frak m}^{n+1}_K\otimes_{O_K} \Omega_{O_K}^{d-2}(\log)$  and $f\in (K')^\times$. 

Assume $\alpha$ is an element as in (a) (resp. (b), resp. (c)), and let $\beta\in {\frak m}^{m+1}_K/{\frak m}^{n+1}_K \otimes_{O_K} \Omega^{d-1}_{O_K}(\log)$ be $Tr_{K'/K}(f) \omega$ (resp. $\omega\wedge d(Tr_{K'/K}(f))$, resp. $\omega \wedge d\log(N(f))$ where $N$ is  the norm map $(K')^\times \to K^\times$). Then
$R_{K'}(\Rsw(\chi_{K'}\wedge \alpha))=\chi_{K'}(E(\alpha)) = \chi(N(E(\alpha))$ (here $N$ is the norm map) $= \chi(E(\beta))= R_K(\Rsw(\chi)\wedge \beta)$. Hence, for the proof of (\ref{R=R}), it is sufficient to prove 

(A) $R_{K'}(f \omega)= R_K(Tr_{K'/K}(f) \omega)$ for $f\in {\frak m}^{1-n'+m'}_{K'}$ and $\omega\in \Omega^d_{O_K}(\log)$,

(B) $R_{K'}(\omega \wedge df)= R_K(\omega \wedge d(Tr_{K'/K}(f)))$ for $f$ as in (A) and $\omega\in \Omega^{d-1}_{O_K}(\log)$, 

(C) $R_{K'}(\omega \wedge d\log(f)) = R_K(\omega\wedge d\log(N(f)))$ for $\omega\in {\frak m}^{1-n+m}_K \otimes_{O_K} \Omega^{d-1}_{O_K}(\log)$ and for $f\in (K')^\times$. 

If $K$ is of characteristic $p$, ${\frak m}^{1-n+m}_K/{\frak m}_K \otimes_{O_K} \Omega^d_{O_K}(\log)$ is a subquotient of $\Omega_K^d$ and $R_K$ coincides with the map induced from $\Res_K: \Omega^d_K\to \F_p$ on the subquotient, and the same thing holds for $K'$. Hence (A), (B), (C) follow from $\Res_{K'}= \Res_K \circ Tr_{K'/K}$ (\ref{rdl}). 

Assume $K$ is of characteristic $0$.  Let $L= F((T))$. Then via the isomorphism $F[T]/(T^{n-m}) \overset{\cong}\to O_K/{\frak m}_K^{n-m}$ of rings with log structures ((\ref{isoTi}) with $a=n-m$), ${\frak m}^{1-n+m}_K/{\frak m}_K \otimes_{O_K} \Omega^d_{O_K}(\log)$ is a subquotient of $\Omega_L^d$ and $R_K$ coincides with the map induced from $\Res_L: \Omega^d_L\to \F_p$ on the subquotient. As is easily seen, there is an extension $L'$ of $L$ of degree $[K':K]$ and an isomorphism $O_{K'}/{\frak m}^{n-m}_K O_{K'}\cong O_{L'}/{\frak m}^{n-m}_{L'} O_{L'}$ of rings with log structures extending the isomorphism  $O_K/{\frak m}^{n-m}_K\cong O_L/{\frak m}^{n-m}_L$ of rings with log structures, and $R_{K'}: {\frak m}_{K'}^{1-n'+m'}\otimes_{O_{K'}} \Omega^d_{O_{K'}}(\log) \to \F_p$ is identified with the map induced from $\Res_{L'}: \Omega^d_{L'} \to \F_p$ on the subquotient. Hence (A), (B), (C) follow from $\Res_{L'}= \Res_L \circ Tr_{L'/L}$ (\ref{rdl}).

\end{sbpara}

\begin{sbpara}\label{pflem3ii} We prove the case (ii) of \ref{lem3}. 
 We have a  commutative diagram
$$\begin{matrix}  
{\hat K}_{d+1}^M(K')  &\to & \Gal((K')^{\ab}/K')\\
\downarrow && \downarrow\\
{\hat K}_d^M(K) &\to& \Gal(K^{\ab}/K)
\end{matrix}$$
where $\hat K_d^M(K)=\varprojlim_i K^M_d(K)/U^iK^M_d(K)$, ${\hat K}^M_{d+1}(K')$ is defined similarly, and the left vertical arrow is the residue homomorphism in \cite{KK1} and the right vertical arrow is the natural one.

Let $R_K: {\frak m}^{1-n+m}_K/{\frak m}_K \otimes_{O_K} \Omega^d_{O_K}(\log)\to \F_p$ and $R_{K'}:  {\frak m}^{1-n+m}_{K'}/{\frak m}_{K'} \otimes_{O_{K'}} \Omega^{d+1}_{O_{K'}}(\log)\to \F_p$ be  the maps defined as in \ref{pflem3i}. 

Let $\chi\in F_nH^1(K, \Q_p/\Z_p)$.  We prove $\Rsw(\chi_{K'})=\Rsw(\chi)_{K'}$. For this, it is sufficient to prove that
$R_{K'}(\Rsw(\chi_{K'})\wedge \alpha)= R_{K'}(\Rsw(\chi)_{K'}\wedge \alpha)$
for any $\alpha\in {\frak m}^{m+1}_{K'}/{\frak m}^{n+1}_{K'}\otimes_{O_{K'}} \Omega^d_{O_{K'}}(\log)$. The last group is generated additively by elements as in the following (i)--(iii). 

(a)  Elements in the image of ${\frak m}_K^{m+1}/{\frak m}_K^{n+1} \otimes_{O_K}  \Omega^d_{O_K[[T]]}\oplus {\frak m}_{K'}^{m+1}/{\frak m}_{K'}^{n+1}\otimes_{O_K} \Omega^d_{O_K}(\log)$. 

(b)  $\omega \otimes d\log(T)$ with $\omega\in {\frak m}^{m+1}_K/{\frak m}_K^{n+1}\otimes_{O_K} \Omega^{d-1}_{O_K}(\log)$.

(c)  $T^{-i}\omega \otimes d\log(T)$ with $\omega\in {\frak m}^{m+1}_K/{\frak m}_K^{n+1}\otimes_{O_K} \Omega^{d-1}_{O_K}(\log)$ and with $i\geq 1$.

Assume $\alpha$ is an element as in (a) (resp. (b), resp. (c)). Let $\Res_{K'/K}: {\hat K}^M_{d+1}(K') \to {\hat K}^M_d(K)$ be the residue map. Then $R_{K'}(\Rsw(\chi_{K'})\wedge \alpha)= \chi_{K'}(E(\alpha))= \chi(\Res_{K'/K}(E(\alpha)))$. By (ii) of \ref{exponential} and by  the definition of $\Res_{K'/K}$ in \cite{KK1}, we have $ \chi(\Res_{K'/K}(E(\alpha)))= \chi(E(\beta)) = R_K(\Rsw(\chi)\wedge \beta)$ where $\beta=0$ in cases (a) and (c) and $\beta= \omega$ in the case (b). Hence it is sufficient to prove 
\begin{equation}
R_{K'}(\sum_{i\gg -\infty}  T^i\omega_i \wedge d\log(T)) = R_K(\omega_0)\quad (\omega_i\in {\frak m}^{1-n+m}_K/{\frak m}_K\otimes_{O_K} \Omega^d_{O_K}(\log)).
\label{R=R2}
\end{equation}

Take a lifting $\iota:F\to O_K/{\frak m}^{n-m}_K$ and extend it to $\iota: F((T))\to O_{K'}\; ;\;\sum_{i\gg -\infty}  a_i T^i\mapsto \sum_{i\gg -\infty} \iota(a_i)T^i$. Let $\pi$ be a prime element of $K$. Write $$\omega_i= \sum_{j=0}^{n-m-1} \pi^{-j}\iota(\omega_{ij})\wedge d\log(\pi) \quad (\omega_{ij}\in \Omega^{d-1}_F).$$
We have 
$$R_{K'}(\sum_{i\gg -\infty}  T^i\omega_i \wedge d\log(T))= R_{K'}(\iota(\sum_{i, j} T^i\omega_{ij})\pi^{-j}d\log(\pi) \wedge d\log(T))
$$ $$= \Res_{F((T))}(\sum_i T^i\omega_{i0}\wedge d\log(T)) = \Res_F(\omega_{00})=R_K(\omega_0).$$ This proves (\ref{R=R2}).

\end{sbpara}

\subsection{Application of higher dimensional global class field theory}\label{global}

\begin{sbpara}

Let $X$ be a proper normal integral scheme  over $\Z$ with function field $J$ and let $\nu$ be a point of $X$ of codimension one whose residue field is of characteristic $p>0$. We assume $X(\R)=\emptyset$. 

 Let $n\geq 1$ and let $F_{\nu, n}H^1(J, \Q_p/\Z_p)\subset H^1(J, \Q_p/\Z_p)$ be the inverse image of $F_nH^1(\hat J_{\nu}, \Q_p/\Z_p)\subset H^1(\hat J_{\nu}, \Q_p/\Z_p)$ where $\hat J_{\nu}$ denotes the field of fractions of the completion of the discrete valuation ring $O_{X,\nu}$. Let ${\frak m}_{\nu}$ be the maximal ideal of $\cO_{X,\nu}$. Let $m=\max(n-e_{\hat J}, [n/p])$.

We  define a canonical homomorphism 
$$\Rsw_{X,\nu}: F_{\nu, n}H^1(J, \Q_p/\Z_p) \to {\frak m}_{\nu}^{-n}/{\frak m}_{\nu}^{-m}\otimes_{O_{X,\nu}} \Omega^1_{X,\nu}(\log)$$
by using the higher dimensional global class field theory in \cite{KSh}. 

\end{sbpara}

\begin{sbpara}

By the class field theory of $X$ \cite{KSh}, we have a canonical continuous homomorphism $$C_X:= \varprojlim_I H^d(X, K^M_d(\cO_X, I))\to \Gal(J^{\ab}/J)$$
where $d=\dim(X)$, and hence
$\chi\in H^1(J, \Q_p/\Z_p)$
 induces a homomorphism $C_X\to \Q_p/\Z_p$ which factors through $H^d(X, K^M_d(\cO_X, I))$ for some $I$ (\ref{gl1}). 

\end{sbpara}

\begin{sbpara} Let $Y$ be the closure of $\nu$ in $X$. We will identify an element $\frak p=(y_i)_i\in P(Y)$ with the element $(x_i)_i\in P(X)$ where  $x_i=y_i$ for $0\leq i\leq d-1$ and $x_d$ is the generic point of $X$. 

Let $\chi\in F_{\nu, n}H^1(J, \Q_p/Z_p)$. 
For $\frak p\in P(X)$, let $h_{\frak p}: K^M_d(J)\to \Q_p/\Z_p$ be the homomorphism induced by $\chi:C_X\to \Q_p/\Z_p$ (\ref{hp}). Then if $\frak p\in P(Y)$, $h_{\frak p}$ kills $U^{n+1}K_d^M(J)$ where $U^{\bullet}$ is defined with respect to the discrete valuation ring $\cO_{X, \nu}$.  For $\frak p\in P(Y)$, let $s_{\frak p}: {\frak m}^{m+1}_{\nu}/{\frak m}^{n+1}_{\nu}\otimes_{\cO_{X,\nu}} \Omega^{d-1}_{X,\nu}(\log)\to \F_p$ be the homomorphism induced by $h_{\frak p}$ and the truncated exponential map. 

\end{sbpara}

\begin{sblem}\label{scoh} There exists a sheaf $\cF$ on $Y$ satisfying the following (i) and (ii).

(i) $\cF$ is a coherent $\cO_X$-submodule of the  constant sheaf ${\frak m}^{m+1}_{\nu}/{\frak m}^{n+1}_{\nu}\otimes_{\cO_{X,\nu}} \Omega^{d-1}_{X,\nu}(\log)$ on $Y$ and 
the map $\cF_{\nu}\to {\frak m}^{m+1}_{\nu}/{\frak m}^{n+1}_{\nu}\otimes_{\cO_{X,\nu}} \Omega^{d-1}_{X,\nu}(\log)$ is an isomorphism. 

(ii) $(s_{\frak p})_{\frak p\in P(Y)}$ defines a homomorphism $H^{d-1}(Y, \cF) \to \F_p$.  

\end{sblem}

\begin{pf} Let $U$ be a regular dense open subset of $X$
such that $D:=U\cap Y$ is also a regular dense open
subset of $Y$ and that $\chi$ is unramified on $U\smallsetminus D$.
Let $\cF_U$ be the coherent $\cO_U$-module $\cO_U(-(m+1)Y)/\cO_U(-(n+1)Y)\otimes_{\cO_U} \Omega^{d-1}_U(\log D)$. Note that $\cF_{U, \nu}={\frak m}^{m+1}_{\nu}/{\frak m}^{n+1}_{\nu}\otimes_{\cO_{X,\nu}} \Omega^{d-1}_{X,\nu}(\log)$.
 By \ref{coh}, it is sufficient to prove the following (i) and (ii). Let $\xi\in Y$ be a point of codimension one. 

(i) If $\xi\in U$,  then for any $\frak p=(y_i)_i\in P(Y)$ such that $y_{d-2}=\xi$, $s_{\frak p}$ kills the image of $\cF_{U, \xi}\to \cF_{U,\nu}$.

(Note that there are only finitely many $\xi\in Y$ of codimension one such that $\xi\notin U$.)

(ii) There is a finitely generated $\cO_{X, \xi}$-submodule $\cF_{\xi}$ of $\cF_{\nu}:= {\frak m}_\nu^{m+1}/{\frak m}_\nu^{n+1}\otimes_{\cO_{X,\nu}} \Omega^{d-1}_{X,\nu}(\log)$ which generates $\cF_{\nu}$ over $\cO_{X,\nu}$ having the following property: 
For any $\frak p=(y_i)_i\in P(Y)$ such that $y_{d-2}=\xi$, $s_{\frak p}$ kills $\cF_{\xi}$. 

Note that for any $\frak p=(y_i)_i\in P(Y)$, we have $\sum_{\frak p'\in R(\frak p)} h_{\frak p'}=0$ on $K^M_d(J)$ where $R(\frak p)$ denotes the set of $\frak p'=(x_i)_i\in P(X)$ such that $x_i=y_i$ for $0\leq j\leq d-2$. 

We prove (i). Let $\pi$ be an element of $\cO_{X,\xi}$ which defines $Y$ at $\xi$. For $\frak p$ as in (i) and for $g\in \cO_{X,\xi}$ and $u_1, \dots, u_{d-1}\in\cO_{X,\xi}^\times \cdot \pi^{\Z}$, $s_{\frak p}(\pi^{m+1} g d\log(u_1)\wedge\dots \wedge  d\log(u_{d-1}))= h_{\frak p}(\{E(\pi^{m+1}g), u_1, \dots, u_{d-1}\})$. This $\{E(\pi^{m+1}g), u_1, \dots, u_{d-1}\}\in K^M_d(J)/U^{n+1}K^M_d(J)$ belongs to the subgroup of $K^M_d(J)/U^{n+1}K^M_d(J)$ generated by all elements of the form $\{u_1,\dots, u_d\}\in K^M_d(J)$ such that $u_1, \dots, u_d\in \cO_{X,\xi}^\times\cdot  \pi^{\Z}$. For any $\frak p'=(x_i)_i\in R(\frak p)\smallsetminus \{\frak p\}$,  $h_{\frak p'}: K^M_d(J)\to \Q_p/\Z_p$  kills such $\{u_1, \dots, u_d\}$ because $\chi$ is unramified at $x_{d-1}$ and hence  $h_{\frak p'}$ factors through the boundary map $K^M_d(J) \to K^M_{d-1}(\kappa(x_{d-1}))$ ($\kappa$ denotes the residue field) which kills $\{u_1,\dots, u_d\}$. Hence $h_{\frak p}(\{u_1,\dots, u_d\})=0$.

We prove (ii). Take an element $\pi$ of $\cO_{X,\xi}$ which is a prime element of the discrete valuation ring $\cO_{X,\nu}$.  
Take an element $f\in \cO_{X,\xi}$ having the following properties (a1) and (a2).  
(a1) $f$ is a unit in $\cO_{X,\nu}$. (a2) $\text{ord}_{\mu}(\pi^{m+1}f) > \Sw_{\mu}(\chi)$ for any point $\mu$ of $X$ of codimension one such that $\xi\in \overline{\{\mu\}}$ and $\mu\neq \nu$ and such that either $\chi$ ramifies at $\mu$ or  $\text{ord}_{\mu}(\pi)>0$. Let $\cF_{\xi}$ be the $\cO_{X,\xi}$-submodule of $\cF_{\nu}$ generated by the images of $\pi^{m+1}f\otimes\Omega^{d-1}_{X,\xi}$ and $\pi^{m+1}f\otimes \Omega_{X,\xi}^{d-2}\wedge d\log(\pi)$. We prove that $s_{\frak p}$ kills 
 $\cF_{\xi}$. For $g\in \cO_{X,x}$ and $u_1, \dots, u_{d-1}\in\cO_{X,\xi}^\times \cdot \pi^{\Z}$,  $s_{\frak p}(\pi^{m+1}f\otimes  g d\log(u_1)\wedge\dots \wedge  d\log(u_{d-1}))= h_{\frak p}(\alpha)$ where $\alpha=\{E(\pi^{m+1}fg), u_1, \dots, u_{d-1}\}\in K^M_d(J)$.  To prove that $h_{\frak p}(\alpha)=0$,  it is enough to prove that $h_{\frak p'}(\alpha)=0$ for any $\frak p' =(x_i)_i\in R(\frak p)\smallsetminus \{\frak p\}$.  Let $\mu=x_{d-1}$. Assume first that either $\chi$ ramifies at $\mu$ or $\text{ord}_{\mu}(\pi)>0$. 
 Then $\alpha\in U^{s+1}_{\mu}K^M_d(J)$ where $s=\Sw_{\mu}(\chi)$ and $U^{\bullet}_\mu$ is the filtration defined for the discrete valuation ring $\cO_{X,\mu}$. Hence $h_{\frak p'}$ kills $\alpha$. Assume next $\chi$ is unramified at $\mu$ and $\pi$ is a unit at $\mu$. Then $h_{\frak p'}$ 
 factors through the boundary map $K^M_d(J)\to K^M_{d-1}(\kappa(\mu))$ which kills $\alpha$. 
\end{pf}

\begin{sbpara}\label{SG} 

 By Serre-Grothendieck duality of the cohomology of coherent sheaves (\cite{Ha}), we have a canonical isomorphism 
$$\varinjlim_{\cF}  \Hom(H^{d-1}(Y,\cF), \F_p)\cong  {\frak m}_\nu^{-n}/{\frak m}_\nu^{-m} \otimes_{O_{X,\nu}} \Omega^1_{X,\nu}(\log)$$
where $\cF$ ranges over coherent $\cO_X$-modules  $\cF$ as in (i) in \ref{scoh}. (The inductive system is given by making $\cF$ smaller and smaller. If $\cF$ and $\cF'$ are as in (i) in \ref{scoh} and $\cF'\subset \cF$, the canonical map $H^{d-1}(Y, \cF')\to H^{d-1}(Y, \cF)$ is a surjective map of finite abelian groups.) 
By \ref{scoh}, $\chi\in F_{\nu,n}H^1(J, \Q_p/\Z_p)$ gives  an element of the left hand side of this isomorphism, and hence gives an element of ${\frak m}_{\nu}^{-n}/{\frak m}_{\nu}^{-m} \otimes_{O_{X,\nu}}  \Omega^1_{X,\nu}(\log)$. This is our  $\Rsw_{X,\nu}(\chi)$. 

\end{sbpara}

\subsection{Rsw in general}\label{s3.5}

We prove our statements in \ref{Rsw1} and \ref{Rsw4}.

\begin{sblem}\label{gen0} Let $K$ be a complete discrete valuation field whose residue field $F$ is a finitely generated field over $\F_p$, and let $\chi\in H^1(K, \Q_p/\Z_p)$. Then there exist a proper normal integral scheme $X$ over $\Z$ such that $X(\R)=\emptyset$,  a point $\nu$ of $X$ of codimension one, and an isomorphism $\alpha$ between  $O_K$ and  the completion of the local ring $\cO_{X,\nu}$ such that if $J$ denotes the function field of $X$,  $\chi$ comes from $H^1(J, \Q_p/\Z_p)$ via $\alpha$. 

\end{sblem}

\begin{pf} If $K$ is of characteristic $p$, take a proper normal integral scheme $Y$ over $\F_p$ with function field $F$,  let $X= {\bf P}^1_Y$, and let $\nu$ be the generic point of the image of any section $Y\to {\bf P}^1_Y$. Next assume  $K$ is of characteristic $0$. Let $(t_i)_{1\leq i\leq r}$ be a transcendental basis of $F$ over $\F_p$ and let $(T_i)_{1\leq i\leq r}$ be its lifting to $O_K$. Let $J_0=\Q(T_1, \dots, T_r)$. Then 
the algebraic closure $J_{\infty}$ of $J_0$ in $K$ is a henselian discrete valuation field whose completion is $K$, and hence $H^1(J_{\infty}, \Q_p/\Z_p)\overset{\cong}\to H^1(K, \Q_p/\Z_p)$. Hence $\chi$ comes from $H^1(J, \Q_p/\Z_p)$ for some finite extension $J$ of $J_0$ in $J_{\infty}$ such that $J$ is dense in $K$. Let $X$ be the integral closure of ${\bf P}^r_{\Z}\supset \Spec(\Z[T_1, \dots, T_r])$ in $J$. By replacing $J$ by a bigger $J$, we have $X(\R)=\emptyset$. (Indeed, $K$ contains a purely imaginary  algebraic number $\beta$ (for example, a  square root  of $1-p$ if  $p\neq 2$, and  a square root of $-7=1-8$ in the case $p=2$).  
By replacing $J$ by $J(\beta)\subset K$, we have $\beta \in J$. For such $J$, $X(\R)=\emptyset$.) Let $\nu$ be image of the closed point of $\Spec(O_K)$ under $\Spec(O_K) \to X$. 
\end{pf}

\begin{sblem}\label{gen1} Let $K$ be a complete discrete valuation field whose residue field $F$ is a function field in $r$ variables over $\F_p$, and let $\chi\in F_nH^1(K, \Q_p/\Z_p)$.  Then there is a unique element  $\Rsw_K(\chi)$ of $ {\frak m}^{-n}_K/{\frak m}_K^{-m} \otimes_{O_K} \Omega^1_{O_K}(\log)$ ($m= \max(n-e_K, [n/p])$) having the following property (i) for any $(X, \nu, \alpha)$ as in \ref{gen0}.

(i) Let $J$ be as in \ref{gen0} and let $Y\subset X$ be the closure of $\nu$. Then for any $v\in Pl(Y)\subset Pl(X)$,
the image of $\Rsw_K(\chi)$ in ${\frak m}^{-n}_{J_v}/{\frak m}_{J_v}^{-m} \otimes_{O_{J_v}} \Omega^1_{O_{J_v}}(\log)$ coincides with the element $\Rsw_{J_v}(\chi_{J_v})$ defined in Section 3.3. Here we regard $K$ as a subfield of $J_v$ via $\alpha$.

\end{sblem}

The proof of \ref{gen1} is given after preparations  \ref{FFv} and \ref{inj}. 

\begin{sblem}\label{FFv} Let $Y$ be an integral scheme over $\F_p$ of finite type, let $F$ be the function field of $Y$, and let $v\in Pl(Y)$. Then we have an isomorphism 
$$F\otimes_F F_v\overset{\cong}\to  F_v\;;\;x\otimes y\mapsto xy^p,$$
where  $F\to F$ in the tensor product is $x\mapsto x^p$.

\end{sblem}

\begin{pf}  This follows from 3 of \ref{cpl9}. \end{pf}

\begin{sblem}\label{inj} Let $(X, \nu, \alpha,J)$ be as in \ref{gen0} and let $Y\subset X$ be the closure of $\nu$. Then  for any $v\in Pl(Y)\subset Pl(X)$ and for $t\geq 1$, the map 
$$O_K/{\frak m}_K^t\otimes_{O_K} \Omega^1_{O_K}(\log) \to O_{J_v}/{\frak m}_{J_v}^t\otimes_{O_{J_v}} \Omega^1_{O_{J_v}}(\log)$$
is injective. More precisely, the map $O_{J_v}\otimes_{O_K} (\text{l.h.s.})\to (\text{r.h.s.})$ is bijective.

\end{sblem}

\begin{pf} This follows from  \ref{FFv}. 
\end{pf}

\begin{sbpara} We prove \ref{gen1}.  Take $(X, \nu, \alpha,J)$ as in \ref{gen0} and let $Y\subset X$ be the closure of $\nu$. 

By Section \ref{global}, 
 $\chi$ defines an element of ${\frak m}_K^{-n}/{\frak m}_K^{-m}\otimes_{O_K} \Omega^1_{O_K}(\log)$ 
which we denote by $\Rsw_K(\chi)$. By the construction of this element in Section \ref{global}, it is sent to $\Rsw_{J_v}(\chi_{J_v})$ for any $v\in Pl(Y)\subset Pl(X)$. 

If we change $(X, \nu, \alpha)$ by $(X', \nu', \alpha')$, since the associated $Y$ and $Y'$ are birational, there is $v\in Pl(Y)$ which belongs also to $Pl(Y')$. By \cite{KKIII} Lemma 1,  there is a unique $K$-isomorphism of complete discrete valuation fields between $J_v$ and $J'_v$ which induces the identity map of the residue field $F_v$. 
By the injectivity \ref{inj}, we have that 
$\Res_K(\chi)$ defined by using $(X,\nu, \alpha)$ coincides with that defined by using $(X', \nu', \alpha')$.

\end{sbpara}

\begin{sbpara}\label{union} 

 Now we prove the unique existence of the definition of $\Rsw$ satisfying (i) and (ii) in  \ref{Rsw4}.

Note that $K=\cup_J J$ where $J$ ranges over subfields of $K$ which are  finitely generated over the prime field. For a sufficiently large such $J$, $J$ contains a prime element of $K$ and $\chi$ comes from  $F_nH^1(\hat J, \Q_p/\Z_p)$ where $\hat J$ denotes the completion of $J$ for the restriction of the discrete valuation of $K$ to $J$. By \ref{gen1} applied to $\hat J$, we get an element of ${\frak m}_{\hat J}^{-n}/{\frak m}_{\hat J}^{-m}\otimes_{O_{\hat J}} \Omega^1_{O_{\hat J}}(\log)$,  and hence an element $\Rsw(\chi)$ of 
${\frak m}_K^{-n}/{\frak m}_K^{-m}\otimes_{O_K} \Omega^1_{O_K}(\log)$ as the image of it. 

This element $\Rsw(\chi)$  is independent of such $J$. This is reduced, by the injectivity \ref{inj}, to  \ref{lem3}. 

It is clear that this $\Rsw$ has the properties (i) and (ii) in \ref{Rsw4}, and the uniqueness follows from \ref{inj}.
 
  \end{sbpara}

\begin{sblem}\label{pmK} Let $n$ and $m$ be as in \ref{Rsw1}. Then for an integer $i$ such that $m<i<n$, $\Rsw: F_nH^1(K, \Q_p/\Z_p) \to {\frak m}_K^{-n}/{\frak m}^{-m}_K \otimes_{O_K} \Omega^1_{O_K}(\log)$ sends $F_iH^1(K, \Q_p/\Z_p)$ to ${\frak m}_K^{-i}/{\frak m}^{-m}_K \otimes_{O_K} \Omega^1_{O_K}(\log)$, and the induced homomorphism 
 $F_iH^1(K, \Q_p/\Z_p)/F_{i-1}H^1(K, \Q_p/\Z_p) \to {\frak m}_K^{-i}/{\frak m}_K^{-i+1}\otimes_{O_K} \Omega^1_{O_K}(\log)$ coincides with the refined Swan conductor mod ${\frak m}_K$ (\ref{eqrsw}). 
 \end{sblem}

\begin{pf} This follows from the definition of the refined Swan conductor mod ${\frak m}_K$ in \cite{KKs}.  \end{pf}

\begin{sbprop}\label{m<i<n} Let $n$ and $m$ be as in \ref{Rsw1}.  

1. The map $\Rsw: F_nH^1(K, \Q_p/\Z_p)/F_mH^1(K, \Q_p/\Z_p)\to {\frak m}_K^{-n}/{\frak m}_K^{-m} \otimes_{O_K} \Omega^1_{O_K}(\log)$ is injective.

2. Let  $\chi\in F_nH^1(K, \Q_p/\Z_p)$ and let $i$ be an integer such that $m<i\leq n$. Assume that  $\Rsw(\chi) \in {\frak m}_K^{-n}/{\frak m}_K^{-m} \otimes_{O_K} \Omega^1_{O_K}(\log)$ belongs to the image of ${\frak m}_K^{-i} \otimes_{O_K} \Omega_{O_K}^1(\log)$ but not to the image of  ${\frak m}_K^{-i+1} \otimes_{O_K} \Omega_{O_K}^1(\log)$. Then $\Sw(\chi) =i$. 

3. The image of $\Rsw$ is contained in the kernel of 

\begin{equation} 
d: {\frak m}_K^{-n}/{\frak m}_K^{-m} \otimes_{O_K} \Omega^1_{O_K}(\log)\to {\frak m}_K^{-n}/{\frak m}_K^{-m} \otimes_{O_K} \Omega^2_{O_K}(\log).
\label{eqd2}
\end{equation}

\end{sbprop} 
\begin{pf} 1 and 2 follow from \ref{pmK} and the injectivity of the refined Swan conductor mod ${\frak m}_K$ (\ref{eqrsw}). 

3 is reduced to the case $F$ is a higher dimensional local field. Then it is reduced to 3 of \ref{Kexp}  by the fact that the map (\ref{eqd2}) is  dual to the map (\ref{eqd1}) in 3 of \ref{Kexp}. 
\end{pf}

In the case of characteristic $p$, the two definitions of the refined Swan conductor mod $p$ coincide:

\begin{sbprop}\label{charp=}  In the case $K$ is of characteristic $p$, the refined Swan conductor mod $p$ defined in Section 3.2 coincides with that defined in Sections 3.3--3.5. 
\end{sbprop}

\begin{sbpara}\label{charpd}

To prove \ref{charp=}, we review the definition of the reciprocity map $K^M_d(K)\to \Gal(K^{\ab}/K)$ of a $d$-dimensional local field of characteristic $p$. 

Let $K$ be a field of characteristic $p>0$, let $s\geq 1$,  and let $P$ be a commutative ring over $\Z/p^s\Z$ which is flat over $\Z/p^s\Z$ such that $P/pP=K$. We have a well-defined ring homomorphism $$\theta: W_s(K)\to P\;;\; (a_{s-1},\dots, a_0)\mapsto \sum_i   p^{s-1-i} \tilde a_i^{p^i}$$
where $\tilde a_i$ is a lifting of $a_i$ to $P$. In $\Omega^1_P$, we have 
\begin{equation} 
d\theta(a_{s-1}, \dots, a_0)= p^{s-1}\sum_i {\tilde a}_i^{p^i-1}d{\tilde a}_i. 
\label{Wittd}
\end{equation}

Now let $K$ be a $d$-dimensional local field of characteristic $p>0$.  The reciprocity map $K^M_d(K)\to \Gal(K^{\ab}/K)$ is defined as follows (\cite{KK1} Chap. 3). Take an isomorphism 
\begin{equation}
K\cong \F_q((T_1))\dots ((T_d))
\label{dlisom}
\end{equation}
 of $d$-dimensional local fields and identify them.
Define rings $P_i$ ($0\leq i \leq d$) inductively as $P_0=W_s(\F_q)$, $P_i= P_{i-1}[[T_i]][T_i^{-1}]$ ($1\leq i\leq d$). Let $P=P_d$. So $P$ is $\Z/p^s\Z$-flat and $P/pP=K$. 
Let $\Res_P :\Omega^d_P\to \Z/p^s\Z$ be the composition
$$\Omega^d_P =\Omega^d_{P_d}\to \Omega^{d-1}_{P_{d-1}}\to \dots \to \Omega^0_{P_0} =W_s(\F_q) \overset{\text{trace}}\to \Z/p^sZ$$
where $\Omega^i_{P_i} \to \Omega^{i-1}_{P_{i-1}}$ ($1\leq i\leq d$) is the map $$\sum_{i\gg -\infty} T^i\omega_i d\log(T_i) \mapsto \omega_0  \quad (\omega_i\in \Omega^{i-1}_{P_{i-1}}).$$ $\Res_P$ kills the image of $d: \Omega^{d-1}_P \to \Omega^d_P$. 

Then the reciprocity map $K^M_d(K)\to \Gal(K^{\ab}/K)$ (which is independent of the choice of the isomorphism (\ref{dlisom}))  is characterized by the following property: 
 Let  $\chi\in H^1(K, \Z/p^s\Z)$ and assume that $\chi$ is the image of $f\in W_s(K)$. Then $\chi$ sends  
 $\{y_1, \dots, y_d\}\in K^M_d(K)$ ($y_i\in K^\times$) to  $\Res_P(\theta(f)d\log(\tilde y_1)\wedge \dots \wedge d\log(\tilde y_d))$ where $\tilde y_i$ is a filting of $y_i$ to $P$. 

\end{sbpara}

\begin{sbpara}\label{pfcharp=} We prove \ref{charp=}. We may assume that $K$ is a $d$-dimensional local field of characteristic $p$ with $d\geq 1$.

Let $m=\max(n-e_K, [n/p])$ and let $x\in {\frak m}_K^{m+1}$, $y_i\in K^\times$ ($1\leq i\leq d-1$). Let the notation be as in \ref{charpd}.
We have 
\begin{align}
&\chi(\{E(x), y_1, \dots, y_{d-1}\})=
\Res_P(\theta(f) d\log E(\tilde x) \wedge d\log(\tilde y_1) \wedge \dots \wedge d\log(\tilde y_{d-1}))\nonumber \\
&=\Res_P( \theta(f) d\tilde x \wedge d\log(\tilde y_1) \wedge \dots \wedge d\log(\tilde y_{d-1}))\quad (\text{here we used 3 of \ref{exponential}})
\nonumber\\
&= \Res(d (\theta(f)\tilde x d\log(\tilde y_1) \wedge \dots \wedge d\log(\tilde y_{d-1}))) -\Res_P(\tilde x d\theta(f) \wedge d\log(\tilde y_1) \wedge \dots \wedge d\log(\tilde y_{d-1}))\nonumber\\
&=-\Res_P( \tilde x d\theta(f) \wedge d\log(\tilde y_1) \wedge \dots \wedge d\log(\tilde y_{d-1})).\nonumber
\end{align}
Write $f=(f_{s-1}, \dots, f_0)$ with $f_i\in K$. By (\ref{Wittd}), $\Res_P( \tilde x d\theta(f) \wedge d\log(\tilde y_1) \wedge \dots \wedge d\log(\tilde y_{d-1}))$ is the image of $-\Res_F(x\cdot \sum_{i=0}^{s-1} f_i^{p^i-1}df_i \wedge d\log(y_1)\wedge \dots d\log(y_{d-1}))\in \F_p$ under the homomorphism $\F_p \overset{\subset}\to  \Z/p^s\Z$ which sends $1$ to $p^{s-1}$.

\end{sbpara}

\subsection{Applications}

A non-logarithmic version of the following result is obtained in \cite{Bo2}. 
\begin{sbthm}\label{thmB} Let $K$ be a complete discrete valuation field whose  residue field is of characteristic $p>0$, and let $\chi\in H^1(K, \Q_p/\Z_p)$. Then 
$$Sw(\chi)= \sup_L Sw(\chi_L)/e(L/K)$$
where $L$ ranges over cdvf over $K$ with perfect residue fields. 
\end{sbthm}

\begin{pf} Write $\Rsw(\chi)=  \pi^{-n} \otimes (a + c d\log(\pi))$ where $\pi$ is a prime element of $K$, $n=Sw(\chi)$, $a\in \Omega^1_{O_K}(\log)$ and $c\in O_K$. If $c$ is a unit, any $L/K$ with $e(L/K)=1$ (the residue field of $L$ is perfect) satisfies $\Rsw(\chi_L)= \pi^{-n}\otimes c d\log(\pi)$ and hence $Sw(\chi_L)= Sw(\chi)$. Assume  $c$ is not a unit. Write  $a= \sum_{\lam\in \Lambda} a_{\lam} db_{\lam}$ where $(b_{\lam})_{\lam\in \Lambda}$ is a lifting of a $p$-base of $F$ to $O_K$ and $a_{\lam}\in O_K$. Then  $a_{\mu}$ is a unit for some $\mu\in \Lambda$. For an integer  $t\geq 2$, let  $L_t$  be the completion of the extension of $K$ obtained by adding $b_{\lam, n}$ ($\lam\in \Lambda$, $n\geq 0$)   such that $b_{\lam,n}=b_{\lam,n+1}^p$ for any $\lam\in \Lambda$ and $n\geq 0$, 
$b_{0,\lam}=b_{\lam}$ for $\lam\neq \mu$, and $b_{\mu,0}=b_{\mu}+\pi^{1/t}$ for a $t$-th root $\pi^{1/t}$ of $\pi$.
 Then the residue field of $L_t$ is perfect, $\pi^{1/t}$ is a prime element in $L_t$, $\Rsw(\chi_{L_t})=\pi^{-n}\otimes  (a_{\mu}\pi^{1/t}+ ct)d\log(\pi^{1/t})$,  and hence we have $e(L_t/K)^{-1}Sw(\chi_{L_t})= Sw(\chi)- t^{-1}$. 
\end{pf}

We improve the previous result of the second author in \cite{IL}. 

\begin{sbpara}\label{thmIL0}   Let $L/K$ be a separable extension of complete discrete valuation fields whose residue fields are of characteristic $p>0$, and assume that $K$ has perfect residue field. (The extension $L/K$ need not be a finite extension.) 

Let ${\hat \Omega}^1_{O_L/O_K}(\log)$ be the ${\frak m}_L$-adic completion of the cokernel of $O_L\otimes_{O_K} \Omega^1_{O_K}(\log) \to \Omega^1_{O_L}(\log)$. Then by \ref{hatdif}, the $O_L$-torsion part ${\hat \Omega}^1_{O_L/O_K}(\log)_{\text{tor}}$ of ${\hat \Omega}^1_{O_L/O_K}(\log)$ is of finite length as an $O_L$-module.  Let  $\delta_\mathrm{tor}(L/K)\in \Z_{\geq 0}$ be the length of the $O_L$-module ${\hat \Omega}^1_{O_L/O_K}(\log)_{\text{tor}}$. 
\end{sbpara}

\begin{sbthm}\label{thmIL1} Let $L/K$ and   $\delta_\mathrm{tor}(L/K)$ be as in \ref{thmIL0}.  Assume  $$\delta_\mathrm{tor}(L/K)<e_L.$$ If
	$\chi \in H^1(K,\Q_p/\Z_p)$ is such that
	$$\Sw (\chi)> \dfrac{p}{p-1}\dfrac{\delta_\mathrm{tor}(L/K)}{e(L/K)}.$$ Then
	$$\Sw (\chi_L)= e(L/K)\Sw(\chi)-\delta_\mathrm{tor}(L/K).$$
\end{sbthm}

\begin{sbrem} 
In the case $K$ is of characteristic $p$, this result is proved in \cite{IL}. In \cite{IL}, it is stated without assuming $L/K$ is separable, but this was a mistake. 
\end{sbrem}

  When $K$ is of mixed characteristic and $\delta_{\mathrm{tor}}(L/K)<e_L$, Theorem \ref{thmIL1}  gives a stronger result than the following result  obtained in \cite{IL}. 
  
  \begin{sbthm}\label{thmIL2} (\cite{IL})
	Let $L/K$ be an  extension of complete discrete valuation fields of mixed characteristic. Assume that $K$ has perfect residue field of characteristic $p>0$ and  
	$\chi\in H^1(K,\Q_p/\Z_p)$ is such that 
	$$\Sw (\chi) \geq
	\dfrac{2e_K}{p-1}+\dfrac{1}{e(L/K)}+\left\lceil\dfrac{\delta_\mathrm{tor}(L/K)}{e(L/K)}\right\rceil.$$ Then $$\Sw (\chi_L) = e(L/K) \Sw (\chi) - \delta_\mathrm{tor}(L/K).$$
\end{sbthm}

\begin{sbpara} We prove \ref{thmIL1}. Write  $n= \Sw(\chi)$, $n' = e(L/K)n$, $m = \max(n-e_K,[n/p])$, $m'=\max(n'-e_L, [n'/p])$ and $\delta= \delta_\mathrm{tor}(L/K)$.
 We prove first 
\begin{equation}
\Sw (\chi_L)= e(L/K)\Sw(\chi)-\delta\quad \text{if}\;\;\;	
n'-\delta>m'.
\label{ILthm}
\end{equation}
 The map $\Rsw: F_nH^1(K, \Q_p/\Z_p)\to {\frak m}_K^{-n}/{\frak m}_K^{-m}\otimes_{O_K} \Omega^1_{O_K}(\log)$ sends $\chi$ to a generator of the $O_K$-module $ {\frak m}_K^{-n}/{\frak m}_K^{-m}\otimes_{O_K} \Omega^1_{O_K}(\log)$  which is isomorphic to $O_K/{\frak m}_K^a$ for some $a\geq 1$. 
If $n'-\delta >m'$, by \ref{hatdif} and \ref{KLdif}, the image of this generator in ${\frak m}_L^{-n'}/{\frak m}_L^{-m'}\otimes_{O_L} \Omega^1_{O_L}(\log)$ belongs to the image of ${\frak m}_L^{-n+\delta}/{\frak m}_L^{-m'} \otimes_{O_L} \Omega^1_{O_L}(\log)$ but does not belong to the image of ${\frak m}_L^{-n+\delta+1}/{\frak m}_L^{-m'} \otimes_{O_L} \Omega^1_{O_L}(\log)$. Hence by the 
compatibility with pullback ((i) of \ref{Rsw4}) and the relationship between $\Rsw$ and $\Sw$ (2 of \ref{m<i<n}), we obtain (\ref{ILthm}). 

 If 
$$n>\dfrac{p}{p-1}\dfrac{\delta}{e(L/K)},$$ 
then $e(L/K)n - \delta > e(L/K)np^{-1}$ and hence $n'-\delta>[n'/p]$.
Further, since $\delta< e_L$, we have $n'-\delta>n'-e_L$. It follows that $n'-\delta>m'$ and the formula holds. 

\end{sbpara}

\begin{sbpara} In the case $\delta_\mathrm{tor}(L/K)<e_L$, \ref{thmIL1} in the mixed characteristic case is stronger than \ref{thmIL2}. 
 Indeed, we have that 
$$	\dfrac{p}{p-1}\dfrac{\delta_\mathrm{tor}(L/K)}{e(L/K)}< \dfrac{2e_K}{p-1}+\dfrac{1}{e(L/K)}+\left\lceil\dfrac{\delta_\mathrm{tor}(L/K)}{e(L/K)}\right\rceil.
$$
To see that, observe that 
$$ 
\dfrac{2e_K}{p-1}+\dfrac{1}{e(L/K)}+\left\lceil\dfrac{\delta_\mathrm{tor}(L/K)}{e(L/K)}\right\rceil > \dfrac{e_K}{p-1} + \dfrac{1}{p-1}\dfrac{\delta_\mathrm{tor}(L/K)}{e(L/K)} +\dfrac{\delta_\mathrm{tor}(L/K)}{e(L/K)}> \dfrac{p}{p-1}\dfrac{\delta_\mathrm{tor}(L/K)}{e(L/K)},
$$
so we have an improvement.

\end{sbpara}

\section{Globalization}

\subsection{Module theoretic preparations}\label{s:gl}

In this paper, a simple normal crossing divisor has no multiplicity (it is reduced as a scheme).

\begin{sbprop}\label{dualX} Let $X$ be a regular scheme of finite type over $\Z$ of dimension $d$, let $D$ be a simple normal crossing divisor on $X$, let $p$ be a prime number, and let $E$  be an effective divisor on $X$ whose support is contained in $D$ such that  the scheme $E=\Spec(\cO_X/\cO_X(-E))\subset X$ is of characteristic $p$.

1. The $\cO_E$-module $\cO_E \otimes_{\cO_X} \Omega^1_X(\log D)$ is locally free of  rank $d$.

2. The $\cO_E$-module $\cO_E \otimes_{\cO_X} \Omega^d_X(\log D)(E-D)$ is the canonical dualizing module of  $E$ relative to $\F_p$. 

\end{sbprop}

\begin{pf} We may assume either $X$ is a scheme over $\F_p$ or $X$ is flat over $\Z$.

Assume first that $X$ is over $\F_p$. Then $X$ is smooth over $\F_p$ and hence $\Omega^1_X(\log D)$ is locally free of rank $d$ and $\Omega^d_X$ is a dualizing module of $X$ relative to $\F_p$. Hence the $\cO_E$-module $\cO_E\otimes_{\cO_X} \Omega^1_X(\log D)$ is locally free of rank $d$ and $R\cH om_{\cO_X}(\cO_E, \Omega^d_X)[1]= \cO_E\otimes_{\cO_X} \Omega_X^d(E) =\cO_E \otimes_{\cO_X} \Omega^d_X(\log D)(E-D)$ is the dualing module of $E$ relative to $\F_p$. 

 Assume $X$ is flat over $\Z$. This case is reduced to the case $X$ is over $\F_p$ as follows.
Let $x$ be a closed point of $X$ and 
let $(t_i)_{1\leq i\leq d}$ be a regular system of parameters of the regular local ring $\cO_{X,x}$ such that $D$ is defined at $x$ by $t_1\dots t_r$ for some $r$ ($0\leq r\leq d$). 
Then $E$ is defined by $t_1^{a(1)}\dots t_r^{a(r)}$ at $x$ for some $a(i)\in \Z_{\geq 0}$ and we have $p=t_1^{e(1)}\dots t_r^{e(r)}h$ for some $e(i)\in \Z_{\geq 0}$ such that $e(i)\geq a(i)$ ($1\leq i\leq r$) and for some $h\in \cO_{X,x}$.
Replacing $X$ by an open neighborhood of $x$ in $X$, we have an unramified  morphism  $X\to {\bf A}^d_\Z$ given by $(t_i)_{1\leq i\leq d}$. 
Replacing $X$ by an open neighborhood of $x$ in $X$, this morphism factors as $X\overset{f}\to Z\overset{g}\to {\bf A}^d_\Z$ where $g$ is \'etale and $f$
 is a closed immersion (\cite{Gr}  18.4.7) defined by a section 
$p-T_1^{e(1)}\dots T_r^{e(r)}\tilde h$ of $\cO_Z$ for some lifting $\tilde h$ of $h$ to $\cO_Z$. Hence $E$ is defined in $Z$ by $T_1^{a(1)} \dots T_r^{a(r)}$ and $p$. 
Hence we have an \'etale morphism from $E$ to $E':=\Spec(\F_p[T_1,\dots, T_d]/(T_1^{a(1)}\dots T_r^{a(r)})) $ for which the pullback of $T_i$ is $t_i$. Let $X'= \Spec(\F_p[T_1, \dots, T_d])$ and let $D'$ be the divisor $T_1\dots T_r=0$ on $X'$ with simple normal crossings. Then 
$\cO_E \otimes_{\cO_X} \Omega^1_X(\log D)$ is the pull back of $\cO_{E'} \otimes_{\cO_{X'}} \Omega^1_{X'}(\log D')$. 1 follows from this. Since $\cO_E\otimes_{\cO_X} \Omega^d_X(\log D)(E-D)$ is the pullback of  $\cO_{E'}\otimes_{\cO_{X'}} \Omega^d_{X'}(\log D')(E'-D')$ which is the dualizing module 
of $E'$ relative to $\F_p$ by the \'etale morphism $E\to E'$, it is the dualizing module of $E$ relative to $\F_p$. 
 \end{pf}

\begin{sbpara}\label{HxRes}  Let the assumption and the notation be as in \ref{dualX} and let $\cF=\cO_E \otimes_{\cO_X} \Omega_X^d(\log D)(E-D)$. By local duality of Grothendieck (\cite{Ha} Chap. V, Section 6), $\Hom(H^{d-1}_x(E, \cF), \F_p)$ is isomorphic to the completion $\hat \cO_{E,x}$ of $\cO_{E,x}$. In the correspondence \ref{lcoh} (we take $E$ as $X$ in \ref{lcoh} and we take $d-1$ as $d$ in \ref{lcoh}), $f\in \cO_{E,x}$ corresponds to $(h_{\frak p})_{\frak p\in P_x(E)}$ with $h_{\frak p}(\omega)= \Res_{\frak p}(f\omega)$ where $\Res_{\frak p} :\cF_{x_{d-1}}\to \F_p$ is as follows. Write $\frak p=(x_i)_{0\leq i\leq d-1}$, let 
 $Y\subset X$ be the closure of $x_{d-1}$, let $F$ be the residue field of $x_{d-1}$, and take an integer $b$ such that $p^b \geq \text{ord}_{x_{d-1}}(E)$. Then $$\Res_{\frak p}= \sum_{v\in Pl(Y, \frak p)} \Res_{F_v}\circ R_b.$$ Here $R_b:\cF_{x_{d-1}}\to \Omega^{d-1}_F$ is as in \ref{Rsw3}, $Pl(Y, \frak p)$ is as in \ref{cpl9}, and $\Res_{F_v}$ is as in \ref{rdl}. Note that $\Res_{F_v}\circ R_b$ is independent of the 
 choice of $b$ because $\Res_{F_v}\circ C=\Res_{F_v}$ where $C$ is the Cartier operator.

\end{sbpara}

\begin{sblem}\label{dualint} 
Let  the assumption and the notation be as in \ref{HxRes}. 
Let $x$ be a closed point of $E$ and let $\xi$ be a point of $E$ of codimension one (it is of codimension two in $X$) such that $x\in \overline{\{\xi\}}$. Let $I$ be the finite set of all generic points $\nu$ of $E$ such that $\xi\in \overline{\{\nu\}}$. 
 Let $\frak q=(x_i)_i$ be an element of $Q_{d-1}(E)$ such that $x_0=x$ and $x_{d-2}=\xi$. For $\nu\in I$, let $\frak q(\nu)$ be the element $(x'_i)_i$ of $P(E)$ defined by $x'_i=x_i$ for $0\leq i\leq d-2$ and $x'_{d-1}=\nu$. Regard $\cF_{\xi}$ as a subset of $\oplus_{\nu \in I} \cF_{\nu}$ by the diagonal embedding.

 1.  Let $a\in \oplus_{\nu\in I} \cF_{\nu}$.  Assume that for any $f\in \cO_{E,\xi}$, we have $\sum_{\nu \in I} \Res_{\frak q(\nu)}(fa_{\nu})=0$. Then $a\in \cF_{\xi}$. 
 
 Let $A$ and $B$ be divisors on $X$ such that $E-D=A+B$. Let $i, j$ be integers $\geq 0$ such that $i+j=d$. Let $\cG= \cO_E \otimes_{\cO_X} \Omega^i_X(\log D)(A)$, $\cH= \cO_E \otimes_{\cO_X} \Omega^j_X(\log D)(B)$. Regard $\cG_{\xi}$ as a subset of $\oplus_{\nu\in I} \cG_{\nu}$ and $\cH_{\xi}$ as a subset of $\oplus_{\nu\in I} \cH_{\nu}$ via the diagonal embeddings.

 2. Let $a\in  \oplus_{\nu\in I} \cH_{\nu}$. Assume that for any $b\in \cG_{\xi}$, we have $b\wedge a \in \cF_{\xi}$ in $\oplus_{\nu\in I} \cF_{\nu}$. Then $a\in \cH_{\xi}$. 

 3. Let $a\in \oplus_{\nu \in I} \cH_{\nu}$.  Assume that for any $b\in \cG_{\xi}$, we have $\sum_{\nu \in I} \Res_{\frak q(\nu)}(b \wedge a_{\nu})=0$. Then $a\in \cH_{\xi}$. 

\end{sblem}

\begin{pf} 1 and 2 can be proved by induction on the multiplicity of $E$ at each point of codimension one of $X$.  3 follows from 1 and 2.
\end{pf}

 \begin{sblem}\label{xxi} Let the assumption and the notation be as in \ref{dualX}. 
 For a vector bundle $\cF$ on $E$, for a closed point $x$ of $E$ and for $a\in \oplus_{\nu} \cF_{\nu}$ where $\nu$ ranges over 
 all generic points of $E$ such that $x\in \overline{\{\nu\}}$, $a$ belongs to the diagonal image of $\cF_x$ if and only if for any point $\xi$ of $E$ of codimension one such that $x\in \overline{\{\xi\}}$, the image of $a$ in $\oplus_{\nu} \cF_{\nu}$, where $\nu$ ranges over all generic points of $E$ such that $\xi\in \overline{\{\nu\}}$, belongs to the diagonal image of $\cF_{\xi}$. 
 
 \end{sblem}
 
 \begin{pf} This can be shown by induction on the multiplicities of $E$ at generic points of $D$. 
   \end{pf}

\subsection{Global refined Swan conductor mod $p$}\label{ss:glRsw}

\begin{sbpara}\label{s4X} In this Section \ref{ss:glRsw}, 
let $X$ be a regular scheme of finite type over $\Z$, let $D$ be a divisor on $X$ with simple normal crossings, let $j: U:=X\smallsetminus D\to X$ be the inclusion morphism, and let $p$ be a prime number. Consider the sheaf $R^1j_*(\Q_p/\Z_p)$ on the \'etale site of $X$, where $j_*$ is the direct image functor for the \'etale topology. 

For an effective divisor $N$ on $X$ whose support is contained in $D$, let $F_NR^1j_*(\Q_p/\Z_p)$ be the subsheaf of $R^1j_*(\Q_p/\Z_p)$ consisting of local sections $\chi$ such that the  Swan conductor $\Sw_{\nu}(\chi)$ at any point $\nu$ of codimension one satisfies $\Sw_{\nu}(\chi) \leq \text{ord}_{\nu}(N)$, where $\text{ord}_{\nu}(N)$ denotes the multiplicity of $N$ at $\nu$. 
Let $F_NH^1_{et}(U, \Q_p/\Z_p)\subset H^1_{et}(U, \Q_p/\Z_p)$ be the inverse image of $H^0(X, F_NR^1j_*(\Q_p/\Z_p))$.

\end{sbpara}

\begin{sbthm}\label{thmS} Let $N$ be an effective divisor on $X$ whose support is contained in $D$. Write $N=\sum_P n(P)P$ ($P$ ranges over all generic points of $D$ and $n(P)\in \Z_{\geq 0}$) and let $M= \sum_P m(P)P$ where $m(P)=\max(n(P)-\text{ord}_P(p), [n(P)/p])$ (so $M$ is an effective divisor on $X$ such that $M\leq N$). 
 Then there is a unique homomorphism 
$$\Rsw:  F_NR^1j_*(\Q_p/\Z_p) \to \cO(N)/\cO(M) \otimes_{\cO_X} \Omega^1_X(\log D)$$
which is compatible with Rsw at all  points of $X$ of codimension one in the support of $N-M$. 
\end{sbthm}

\begin{pf} We may assume that $X$ is a dense open subscheme of a proper normal $d$-dimensional integral scheme $\bar X$ over $\Z$ such that $\bar X(\R)=\emptyset$. We use the global class field theory of $\bar X$. (Note that $\bar X$ here is $X$ in Section 3.3.) Let $\chi\in F_NH^1_{et}(U, \Q_p/\Z_p)$. Let $K$ be the function field of $X$ and let $h_{\frak p}: K_d^M(K)\to \Q_p/\Z_p$ for $\frak p\in P(\bar X)$ be the homomorphism induced by $\chi$. Let $Y\subset X$ be  the support of the divisor $E=N-M$.
For $\frak p=(x_i)_i\in  P(Y)$,
let $s_{\frak p}:(\cO_X(-M-D)/\cO_X(-N-D) \otimes_{\cO_{X,\nu}} \Omega^{d-1}_X(\log D))_{\nu}\to \F_p$  be the homomorphism induced by $h_{\frak p}$ and the truncated exponential map.  By \ref{lcoh}, \ref{dualint} and \ref{xxi}, it is sufficient  to prove that (under the notation in \ref{dualint} where we take $A= N$,  $B= -M-D$, $i=1$ and $j=d-1$), for any point $\xi$ of $Y$ of codimension one and for any $\frak q=(x_i)_i\in Q_{d-1}(Y)$ such that $x_{d-2}=\xi$, $\sum_{\nu\in I} s_{\frak q(\nu)}$ kills $L:=(\cO_X(-M-D)/\cO_X(-N-D)\otimes_{\cO_X} \Omega^{d-1}_X(\log D))_{\xi}$, where $I$ and ${\frak q}(\nu)$ is as in \ref{dualint}. This is proved by the following Fact 1 and Fact 2.   

Fact 1. We have  $\sum_{\frak p\in P_{\frak q}(X)} h_{\frak p}=\sum_{\frak p\in P_{\frak q}(\bar X)} h_{\frak p}=0$ on  $K^M_d(K)$, where $\frak q$ is identified with an element of $P_{d-1}(X)$.

Fact 2.  If $\frak p =(x_i)_i\in P_{\frak q}(X)$ and $\mu:= x_{d-1}$ does not belong to $Y$, $h_{\frak p}$ factors through the boundary map $K^M_d(K)\to K_{d-1}^M(\kappa(\mu))$ and $s_{\frak p}(L)$ is generated by elements of the form $h_{\frak p}(\{u_1,\dots, u_d\})$ with $u_i\in \cO_{X,\xi}^\times$ and such $\{u_1,\dots, u_d\}$ is killed by the boundary map.
\end{pf}

\begin{sbprop} Let the notation be as in \ref{thmS}. The map $$\Rsw: F_NR^1j_*(\Q_p/\Z_p)/F_MR^1j_*(\Q_p/\Z_p) \to  \cO_X(N)/\cO_X(M) \otimes_{\cO_X} \Omega^1_X(\log D)$$ is injective, and its image is contained in the kernel of 
$$d: \cO_X(N)/\cO_X(M) \otimes_{\cO_X} \Omega^1_X(\log D)
\to  \cO_X(N)/\cO_X(M) \otimes_{\cO_X} \Omega^2_X(\log D).$$

\end{sbprop}

\begin{pf} This follows from \ref{m<i<n}. \end{pf}

\subsection{Pullbacks}

 We prove that the global Rsw (\ref{thmS}) commutes with the pullback maps. 

\begin{sbthm}\label{pullback} Let $f: X'\to X$ be a morphism of regular schemes of finite type over $\Z$, let $D$ (resp. $D'$) be a divisor on $X$ (resp. $X'$) with simple normal crossings, and assume $f(U')\subset U$ where $U=X\smallsetminus D$, $U'=X'\smallsetminus D'$. Let $p$ be a prime number. Let $N$ be an effective divisor on $X$ whose support is contained in $D$, let $N'$ be the pullback $f^*N$ of $N$ on $X'$, define the effective divisor $M\leq N$ on $X$ using $N$ as in \ref{thmS} and 
define the effective divisor $M'\leq N'$ on $X'$ using $N'$ similarly. Let $j':U'\to X'$ be the inclusion morphism.

1. The map $f^*R^1j_*(\Q_p/\Z_p)\to R^1j'_*(\Q_p/\Z_p)$ sends $f^*F_NR^1j_*(\Q_p/\Z_p)$ to $F_{N'}R^1j'_*(\Q_p/\Z_p)$.

2. We have a commutative diagram
$$\begin{matrix}  
f^*F_NR^1j_*(\Q_p/\Z_p) &\overset{\Rsw}\to & f^*(\cO_X(N)/\cO_X(M)  \otimes_{\cO_X} \Omega^1_X(\log D))\\
\downarrow && \downarrow\\
F_{N'}R^1j'_*(\Q_p/\Z_p) & \overset{\Rsw}\to & \cO_{X'}(N')/\cO_X(M')  \otimes_{\cO_{X'}} \Omega^1_{X'}(\log D').$$
\end{matrix}$$

\end{sbthm}

We first prove Theorem \ref{pullback} in special cases (\ref{special1}, \ref{special}, \ref{special2}).

\begin{sblem}\label{special1} Theorem \ref{pullback} is true if $X'$ is the blowing-up of $X$ at a closed point of $D$. 

\end{sblem}

To prove \ref{special1}, we use

\begin{sblem}\label{X-x0} Let the notation be as in \ref{thmS}.  Let $X'$ be the blowing-up of $X$ at a closed point $x$ of $D$ and let $D'$ be the support of the pullback of $D$ on $X'$.  Assume $\dim(\cO_{X,x})\geq 2$.  Let $\la: X\smallsetminus \{x\}\to X$ be the inclusion map. Then the map 
\begin{equation}
f_*(\cO_{X'}(N')/\cO_{X'}(M') \otimes_{\cO_{X'}} \Omega^1_{X'}(\log D')) \to \la_*\la^*(\cO_X(N)/\cO_X(M)\otimes_{\cO_X} \Omega^1_X(\log D))
\label{X-x}
\end{equation}
induced by the inclusion map $X\smallsetminus \{x\}\to X'$ is injective. 

\end{sblem}

\begin{pf} Let $D_i$ ($1\leq i\leq r$) be the irreducible components of $D$ which contain $x$. Replacing $X$ by an open neighborhood of $x$, we may assume that $D=\cup_{i=1}^r D_i$. Write $N=\sum_{i=1}^r n_iD_i$, $M= \sum_{i=1}^r m_iD_i$ ($n_i, m_i\in \Z_{\geq 0}$). We have $N'=\sum_{i=1}^r n_iD'_i + (\sum_{i=1}^r n_i)P$ where $D'_i$ is the proper transformation of $D_i$ in $X'$ and $P$ is the inverse image of $x$ in $X'$. We have $M'= \sum_{i=1}^r m_iD'_i + m'P$ for some integer $m'$ such that $\sum_{i=1}^r m_i \leq m' \leq \sum_{i=1}^r n_i$. Take integers $a_i$ ($1\leq i\leq r$) such that $m_i \leq a_i\leq n_i$ and $\sum_{i=1}^r a_i=m'$, and let $A$ be the divisor $\sum_{i=1}^r a_iD_i$ on $X$. Then the pullback $A'$ of $A$ on $X'$ is $\sum_{i=1}^r a_iD'_i+m'P$. It is sufficient to prove that the maps
\begin{equation}
f_*(\cO_{X'}(N')/\cO_{X'}(A') \otimes_{\cO_{X'}} \Omega^1_{X'}(\log D')) \to \la_*\la^*(\cO_X(N)/\cO_X(A) \otimes_{\cO_X} \Omega^1_X(\log D))
\label{X-x1}
\end{equation}
\begin{equation}
f_*(\cO_{X'}(A')/\cO_{X'}(M') \otimes_{\cO_{X'}} \Omega^1_{X'}(\log D')) \to \la_*\la^*(\cO_X(A)/\cO_X(M) \otimes_{\cO_X} \Omega^1_X(\log D))
\label{X-x2}
\end{equation}
are injective. 

We first prove that the map (\ref{X-x2}) is injective. Regard the divisors $A-M$ and $A'-M'$ as schemes. Then $(A-M)\smallsetminus \{x\}$ is a dense open subscheme of $A'-M'$, $\cF:= \cO_{X'}(A')/\cO_{X'}(M') \otimes_{\cO_{X'}} \Omega^1_{X'}(\log D')$ is a vector bundle on $A'-M'$, and the map (\ref{X-x2}) is the canonical map $\cF\to \gamma_*\gamma^*\cF$ for the inclusion map $\gamma: (A-M)-\{x\} \to A-M$ and hence it is injective.

We next prove that the map (\ref{X-x1}) is injective. Take a sequence of divisors $A_i$ ($0\leq i\leq s$) on $X$ such that $N=A_s\geq \dots \geq A_0=A$ and such that for each $1\leq i\leq s$, $A_i-A_{i-1}$ coincides with the divisor $D_j$ for some $j$ ($1\leq j\leq r$, $j$ can depend on $i$, $D_j$ is regarded as a reduced scheme), and let $A_i'$ be the pullback of $A_i$ to $X'$. It is sufficient to prove that the map
\begin{equation}
 f_*(\cO_{X'}(A'_i)/\cO_{X'}(A'_{i-1}) \otimes_{\cO_{X'}} \Omega^1_{X'}(\log D')) \to \la_*\la^*(\cO_X(A_i)/\cO_X(A_{i-1}) \otimes_{\cO_X} \Omega^1_X(\log D))
\label{X-x3}
\end{equation}
is injective for each $1\leq i\leq s$. Since $\cO_{X'}(A'_i)/\cO_{X'}(A'_{i-1}) \cong \cO_Q$ where $Q= D_j'\cup P$ with the reduced scheme structure for some $j$, it is sufficient to prove that the map 
\begin{equation}
f_*(\cO_Q \otimes_{\cO_{X'}} \Omega^1_{X'}(\log D')) \to \la_*\la^*(\cO_{D_j} \otimes_{\cO_X} \Omega^1_X(\log D))
\label{X-x4}
\end{equation}
is injective. We may assume that $r=d$ where $d$ is the dimension of $\cO_{X,x}$. In fact, on an open neighborhood of $x$, we can enlarge $D$ to a simple normal crossing divisor $D^*$ with $d$ irreducible components such that $x$ is contained in all  irreducible components of $D^*$, and we can replace $D$  by $D^*$ because the map $\cO_Q \otimes_{\cO_{X'}} \Omega^1_{X'}(\log D')\to \cO_Q \otimes_{\cO_{X'}} \Omega^1_{X'}(\log (D^*)')$ is injective where $(D^*)'$ is the support of the pullback of $D^*$ to $X'$. 
In the case $r=d$, $\Omega^1_{X'}(\log D')$ is the pullback of $\Omega^1_X(\log D)$, and hence the injectivity of (\ref{X-x4}) is reduced to the injectivity of 
\begin{equation}
f_*(\cO_Q) \to \la_*\la^*(\cO_{D_j}).
\label{X-x5}
\end{equation}
Since $P$ is isomorphic to the $d-1$-dimensional projective space over the residue field of $\kappa(x)$, $f_*(\cO_P)$ consists of constant functions. Hence (\ref{X-x5}) is injective. 
\end{pf}

\begin{sbpara} We prove \ref{special1}. Let $U'$ be the inverse image of $U$ in $X'$ (so we have an isomorphism $U'\overset{\cong}\to U$). Let $\chi\in F_NH^1_{et}(U, \Q_p/\Z_p)$. By \cite{KKs} Thm. 8.1, the Swan conductor divisor of $\chi_{U'}$ on $X'$ is $\leq N'$. 
Both $\Rsw(\chi_{U'})$ and the pullback of $\Rsw(\chi)$ on $X'$ are sections of the left hand side of (\ref{X-x}) whose images in  the right hand side of (\ref{X-x}) coincide with the pullback of $\Rsw(\chi)$. Hence by the injectivity \ref{X-x0}, they coincide.

\end{sbpara}

\begin{sblem}\label{special}
Theorem \ref{pullback} is true if $X'$ is a (locally closed) subscheme of $X$ of codimension one, $D$ is regular,  and the scheme    $D\times_X X'$ is regular.

\end{sblem}

The following proof of \ref{special} models on the method of Brylinski in \cite{Br}  who studied the induced ramification on a curve $X'\subset X$ for a surface $X$ over a finite field by using the class field theory of $X$. 

\begin{pf} Let $D'=D\times_X X'$.
We may assume $X$, $D$, $X'$ and $D'$ are integral. Let $K$ (resp. $F$, resp. $K'$, resp.  $F'$) be the function field of $X$ (resp. $D$, resp. $X'$,  resp. $D'$). Let $d$ be the dimension of $X$. 
Let $\xi$ be the generic point of $D'$ and let $\nu$ be the generic point of $D$. 
Let $\tau\in \cO_{X,\xi}$ be an element which defines $X'$ at $\xi$ and let $\pi\in \cO_{X,\xi}$ be an element which 
defines $D$ at $\xi$. Let $\frak q=(x_i)_i \in Q_{d-1}(X)$ and assume $x_{d-2}=\xi$. Let $\chi\in F_NH^1_{et}(U, \Q_p/\Z_p)$. For $\frak p\in P_{\frak q}(X)= P_{\frak q}(\bar X)$, consider the homomorphism $h_{\frak p}: K^M_d(K) \to \Q_p/\Z_p$ induced by $\chi$. Then we have $\sum_{\frak p\in P_{\frak q}(X)} h_{\frak p}=0$ on $K_d^M(K)$. Let $\frak p_1=(x_i)_i \in P_{\frak q}(X)$ be the unique element such that $x_{d-1}$ is the generic point of $D$, and let $\frak p_2=(x_i')_i$ be the unique element of $P_{\frak q}(X)$ such that $x'_{d-1}$ is the generic point of $X'$. 

Write $N=nD$, $M=mD$. We may assume $n>0$ and hence $n>m$.
Let  $g\in \pi^{m+1}\cO_{X, \xi}$ and $y_1,\dots, y_{d-2}\in \cO_{X,\xi}^\times \cdot \pi^{\Z}$,
and let  $\alpha= \{E(g), y_1, \dots, y_{d-2}, \tau\}\in K^M_d(K)$.
Since $\sum_{\frak p\in P_{\frak q}(X)} h_{\frak p}(\alpha)=0$ and $h_{\frak p}(\alpha)=0$ for $\frak p\in P_{\frak q}(X)\smallsetminus \{\frak p_1, \frak p_2\}$, we have

 \begin{equation}
h_{\frak p_1}(\alpha)+ h_{\frak p_2}(\alpha)=0. 
\label{p1p21}
\end{equation}

Let $\chi'$ be the pullback of $\chi$ to $X'\cap U$. Let $\frak p_1(D)$ be $\frak p_1$ 
regarded as an element of $P(D)$ and define $\frak p_2(X')\in P(X')$ and $\frak q(D')\in P(D')$ similarly. Let $h_{\frak p_2(X')}: K^M_{d-1}(K') \to \Q/\Z$ be the homomorphism induced by $\chi'$. Then $h_{\frak p_2}: K^M_d(K)\to \Q/\Z$ coincides with the composition $K^M_d(K)\overset{\partial}\to K^M_{d-1}(K')\to \Q/\Z$ where the first arrow is the boundary map and the second arrow is $h_{\frak p_2(X')}$.

We first prove $\Sw_{\xi}(\chi')\leq n$. 
It is sufficient to prove that $h_{\frak p_2(X')}$ kills $U^{n+1}K^M_{d-1}(K')$ where $U^{\bullet}$ is for the discrete valuation ring $\cO_{X',\xi}$. The group $U^{n+1}K^M_{d-1}(K')$ is generated by $\partial(\alpha)$ where  $\alpha=\{E(g), y_1, \dots, y_{d-2}, \tau\}$ is as above such that $g\in 1+\pi^{n+1}\cO_{X,\xi}$. For such $\alpha$, we have $h_{\frak p_2(X')}(\partial(\alpha))= h_{\frak p_2}(\alpha)=-h_{\frak p_1}(\alpha)=0$. Here we used (\ref{p1p21}) and the fact $\Sw_{\nu}(\chi)\leq n$ and hence $h_{\frak p_1}$ kills $\alpha\in U^{n+1}K^M_d(K)$ where $U^{\bullet}$ is for the discrete valuation ring $\cO_{X,\nu}$.

Let $E$ and $E'$ be the schemes $(n-m)D$ and $(n-m)D'$, respectively. 
Since $\F_p[T]\to \cO_{D,\xi}\;;\;T\mapsto \tau$ is a localization of a smooth map, there is a ring homomorphism $\iota: \cO_{D, \xi}\to \cO_{E, \xi}$ which lifts the identity map of $\cO_{D,\xi}$ and which sends the image of $\tau$ to the image of $\tau$. Then $\iota$ induces ring homomorphisms $F\to \cO_{E, \nu}$ and $F'\to \cO_{E',\xi}$ which lift the identity maps of $F$ and $F'$, respectively. Write
$$\Rsw(\chi)= \sum_{i=m+1}^n   \pi^{-i} \otimes (\iota(a_i) d\log(\pi) + \iota(b_i)) \in (\cO_X(nD)/\cO_X(mD)\otimes_{\cO_X} \Omega^1_X(\log D))_{\xi},$$
$$\Rsw(\chi')= \sum_{i=m+1}^n   \pi^{-i} \otimes (\iota(a'_i) d\log(\pi) + \iota(b'_i))\in (\cO_{X'}(nD')/\cO_{X'}(mD')\otimes_{\cO_{X'}} \Omega^1_{X'}(\log D'))_{\xi},$$
where $a_i\in \cO_{D,\xi}, b_i\in \Omega^1_{D,\xi}, a'_i\in F', b_i'\in \Omega_{F'}^1$.  

For $j\in \Z$, let $\epsilon: \Omega^j_{D, \xi} \to \Omega^j_{F'}$ be the canonical projection. 
It is sufficient to prove $a'_i=\epsilon(a_i), b'_i=\epsilon(b_i)$ for all $i$.

We prove $a'_i=\epsilon(a_i)$. For any $c\in \Omega^{d-2}_{D, \xi}$, we have 
$$h_{\frak p_1}(\{E(\pi^i\otimes \iota(c)), \tau\})= (-1)^{d-1}\Res_{\frak p_1(D)}(a_ic \wedge d\log(\tau)) = (-1)^{d-1}\Res_{\frak q(D')}(\epsilon(a_i)\epsilon(c)),$$ 
where $\Res_{\frak p_1(D)}=\sum_{v \in Pl(D, \frak p_1(D))} \Res_{F_v}$, $\Res_{\frak q(D')}= \sum_{v\in Pl(D', \frak q(D'))} Res_{F'_v}$. 
On the other hand, $$h_{\frak p_2}(\{E(\pi^i \otimes  \iota(c)), \tau\})= h_{\frak p_2(X')}(E(\pi^i\otimes \iota(c))) = (-1)^{d-2} \Res_{\frak q(D')}(a_i' \epsilon(c)).$$
Hence by (\ref{p1p21}), $\Res_{\frak q(D')}(\epsilon(a_i)c)=\Res_{\frak q(D')}(a'_ic)$ for any $c\in \Omega^{d-2}_{F'}$. Hence we have $a_i=a_i'$.

We prove $b'_i=\epsilon(b_i)$. For $c\in \Omega^{d-3}_{D, \xi}$, 
$$h_{\frak p_1}(\{E(\pi^i \otimes \iota(c)),\pi,  \tau\})=-\Res_{\frak p_1(D)}(b_i\wedge c \wedge d\log(\tau)) = -\Res_{\frak q(D')}(\epsilon(b_i)\wedge \epsilon(c)).$$
On the other hand, $$h_{\frak p_2}(\{E(\pi^i  \otimes \iota(c)), \pi,  \tau\})=h_{\frak p_2(X')}(\{E(\pi^i\otimes \iota(c)), \pi\}) = \Res_{\frak q(D')}(b_i' \wedge \epsilon(c)).$$
Hence by (\ref{p1p21}), $\Res_{\frak q(D')}(\epsilon(b_i) \wedge c)=\Res_{\frak q(D')}(b'_i\wedge c)$ for any $c\in \Omega^{d-3}_{F'}$. Hence we have $b_i=b_i'$. \end{pf}

\begin{sbcor}\label{special0} Theorem \ref{pullback} is true in the case $X'$ is a one-dimensional subscheme of $X$ such that $X'$ meets $D$ only at regular points of $D$ and the scheme $D\times_X X'$ is regular. 
\end{sbcor}

\begin{pf} This is because the morphism $X'\to X$ is a composition of morphisms $X'\to X$ of the type of \ref{special}. 
\end{pf}

\begin{sblem}\label{special2} Theorem \ref{pullback} is true if $X'$ is one-dimensional. 

\end{sblem}

\begin{pf} 
Let $C$ be the image of $X' \to X$. We may assume that $C$ is one-dimensional. By repeating the blowing-up of $X$ at $C\cap D$, we obtain the situation of \ref{special0} with $X'=C$. Hence in the composition $X'\to C \to X$, \ref{special2} for  the latter morphism  is reduced to \ref{special1} and  \ref{special0}.  \ref{special2} for the first morphism is reduced to (i) of \ref{Rsw4}. 
\end{pf}

The following \ref{dif} is a preparation for \ref{injrest} which is important to reduce Theorem \ref{pullback} to \ref{special2}

\begin{sblem}\label{dif} Let $D$ be a smooth scheme over a perfect field $k$ of characteristic $p>0$ and let $x$ be a closed point of $D$. Then we have an injection $$\Omega^1_{D/k,x}\to \prod_Z \Omega^1_{Z/k,x}$$ where $Z$ ranges over all closed integral subschemes of $D$ of dimension one which contain $x$ and which are smooth at $x$. 

\end{sblem}

 \begin{pf} Let $\omega$ be a non-zero element of $\Omega^1_{D/k,x}$. 
Let $(t_i)_{1\leq i\leq d}$ be a regular system of parameters at $x$. We prove by induction on $d$. Write $\omega= \sum_i g_i dt_i$ on an open neighborhood of  $x$. We may assume $g_1\neq 0$. Assume $d\geq 2$. 
Since  the ideals of $(t_1^n -t_d)$ ($n\geq 1$) of $\cO_{X,x}$  are distinct prime ideals of height one, there is $n\geq 1$ which is divisible by $p$ such that $g_1\notin (t_1^n-t_d)$. Let $Z$ be the closed integral subscheme of $X$ containing $x$ which is defined by $t_1^n-t_d$ at $x$. Then $Z$ is smooth at $x$, $(dt_i)_{1\leq i\leq d-1}$ is a basis of $\Omega^1_{Z,x}$,  the image of $\omega$ in $\Omega^1_{Z,x}$ is $\sum_{i=1}^{d-1} g_idt_i$, and $g_1$ is non-zero in $\cO_{Z,x}$.
 \end{pf}

\begin{sblem}\label{injrest} Let the notation be as in \ref{thmS}. 
 Let $x$ be a closed point of $D$ at which $D$ is regular. Then we have an injection
 $$(\cO_X(N)/\cO_X(M) \otimes_{\cO_X} \Omega^1_X(\log D))_x\to \prod_C (\cO_C(N'_C)/\cO_C(M'_C) \otimes_{\cO_C} \Omega^1_C(\log D'_C))_x$$ where $C$ ranges over all one-dimensional regular integral subschemes of $X$ which contain $x$ and which are not contained in $D$.
 Here $N'_C$ is the pullback of $N$ to $C$, $M'_C$ is the effective divisor on $C$ defined using $N'_C$ similarly to $M$, and $D'_C$ is the support of the pullback of $D$ to $C$ regarded as a divisor on $C$ with simple normal crossings.

\end{sblem}

\begin{pf} We may assume  that $D$ is regular and that $N=n D$ for some integer $n\geq 1$. Let $\omega$ be a non-zero element of $(\cO_X(N)/\cO_X(M) \otimes_{\cO_X} \Omega^1_X(\log D))_x$. 
Let $\pi$ be an element of $\cO_{X,x}$ which defines $D$ at $x$. Take a system of regular parameters $(t_1, \dots, t_d)$ of $\cO_{X,x}$ such that $t_d=\pi$. 
Write $M=mD$. Let $h\geq 0$ be the largest integer such that $\omega$ is in the image of $\pi^{h-n} \otimes \Omega^1_X(\log D)$. Then $n-h>m$. 
Write $\omega= \pi^{h-n}  \otimes (f d\log(\pi) + \sum_{i=1}^{d-1} g_i dt_i)$ with  $f, g_i\in \cO_{X,x}$. 
 Then the pullback of some of $f$ and $g_i$ to $D$ is non-zero. 

There are two cases.

Case 1. The pullback of some $g_i$ to $D$ is non-zero.

Case 2.  the pullbacks of $g_i$ to $D$ are zero for all $i$. In this case, the pullback of $f$ to $D$ is non-zero.

In Case 1 (resp. Case 2), 
  there is a closed integral subscheme $Z$ of $D$ of dimension one which contains $x$ and which is smooth at $x$ such that the pullback of $\sum_{i=1}^{d-1} g_idt_i$ (resp. $f$) to $Z$ is non-zero. (In Case 1, the existence of $Z$ is by \ref{dif}.) By changing the choice of  $(t_i)_{1\leq i\leq d-1}$, we may assume that $Z$ is defined at $x$ by $t_2, \dots, t_{d-1}$ and $\pi$. (In Case 1, the pullback of $\sum_{i=1}^{d-1} g_idt_i$ to $Z$ becomes the pull back of $g_1dt_1$.)

 Let $a$ be the order of the pullback of $g_1t_1$ (resp. $f$) to $Z$ at $x$.

Take an integer $e$ which is divisible by $p$ (resp. is coprime to $p$) such that  $e\geq p(a+1)$. Let $C$ be a one-dimensional regular integral subscheme of $X$ containing $x$ which is defined by $t_2, \dots, t_{d-1}, t_1^e-\pi$ at $x$. Let $\omega'$ be the pullback of  $\omega$ to $(\cO_C(N'_C)/\cO_C(M'_C)\otimes_{\cO_C}\Omega^1_C(\log D'_C))_x$. 

In Case 1 (resp. Case 2), we have 

{\bf Claim 1.} $N'_C=enD'_C$, $M'_C=m'D'_C$ where $m'$ is an integer such that $e^{-1}m'\leq m + p^{-1}(p-1)$. $t_1$ is a prime element of the discrete valuation ring $\cO_{C,x}$. $d\log(t_1)$ is a generator  of $\Omega^1_C(\log D'_C)_x$

{\bf Claim 2.}  $\omega'=\pi^{h-n} \otimes  gd\log(t_1)$, 
where $g$ is the pullback of $g_1t_1$ (resp. $ef+g_1t_1$) to $C$.  The order  of $g$  at $x$ is $a$.

The proofs of Claim 1 and Claim 2 are straightforward.

{\bf Claim 3.}  $\omega'\neq 0$.

To prove Claim 3, by  Claim 1 and Claim 2, it is sufficient to prove $e(n-h)-a> m'$. But $e^{-1}(e(n-h)-a-m'-1) = (n-h-m-1)+ m-e^{-1}m'+1-e^{-1}(a+1) \geq (m-e^{-1}m')+1 -e^{-1}(a+1) \geq -p^{-1}(p-1)+1 - p^{-1}=0$. 
\end{pf}

\begin{sbpara} Now we prove Theorem \ref{pullback}. By  \ref{injrest} which we apply by taking $X'$ in \ref{pullback} as $X$ in \ref{injrest}, we may assume that $X'$ in \ref{pullback} is one dimensional. Then we are reduced to \ref{special2}. 
\end{sbpara}

\begin{sbpara} By \ref{special2} and \ref{injrest}, for $\chi\in F_NH^1_{et}(U, \Q_p/\Z_p)$, $\Rsw(\chi)$ in Theorem \ref{thmS} is determined by $\Rsw$ of the pullbacks of $\chi$ to $F_{n(C,x)}H^1(K_{C,x},\Q_p/\Z_p)$, where $x$ ranges over regular points of $D$, $C$ ranges over one-dimensional regular integral subschemes of $X$ which contain $x$ and which are not contained in $D$, $n(C,x)$ denotes the multiplicity of the pullback of $N$ to $C$ at $x$, and $K_{C,x}$ denotes the local field of the function field of $C$ at $x$.  

\end{sbpara}

 \begin{sbpara} In the case $X$ is of characteristic $p$, the proofs of the theorems \ref{thmS} and \ref{pullback} can be given also by using 
 Artin-Schreier-Witt theory. 
 
    We first consider \ref{thmS}. 
 The exact sequence $0\to \Z/p^s\Z\to W_s(\cO_U)\overset{\phi-1}\to W_s(\cO_U) \to 0$ on $U_{et}$, where $\phi$ is the Frobenius, induces an exact sequence $0\to \Z/p^s\Z\to j_*W_s(\cO_U) \overset{\phi-1}\to j_*W_s(\cO_U) \to R^1j_*(\Z/p^s\Z) \to 0$ on $X_{et}$ (this is because $R^1j_*W_s(\cO_U)=0$ by the fact $j$ is an affine morphism).  For an effective divisor $N$ whose support is contained in $D$, let $F_Nj_*W_s(\cO_U)$ be the subsheaf of $j_*W_s(\cO_U)$ consisting of local sections $(f_{s-1}, \dots, f_0)$ satisfying $p^i\text{div}(f_i) \geq -N$ for all $i$. Let $F'_NR^1j_*(\Z/p^s\Z)\subset R^1j_*(\Z/p^s\Z)$ be the image of $F_Nj_*W_s(\cO_U)$. 
 
 Claim. Let $F_NR^1j_*(\Z/p^s\Z):= R^1j_*(\Z/p^s\Z)\cap F_NR^1j_*(\Q_p/\Z_p)$ where we embed $\Z/p^s\Z$ in $\Q_p/\Z_p$ in the canonical way. Then $F'_NR^1j_*(\Z/p^s\Z)= F_NR^1j_*(\Z/p^s\Z)$. \
  
 Proof of Claim.  Define the effective divisor $M$ on $X$ from $N$ as in \ref{thmS}. We have an exact sequence
 \begin{equation} 
 0\to \Z/p^s\Z\to F_M j_*W_s(\cO_U)  \overset{\phi-1}\to F_N j_*W_s(\cO_U) \to F'_NR^1j_*(\Z/p^s\Z)\to 0
 \label{ASWX}
 \end{equation} 
 By using (\ref{ASWX}), we can show that for an effective divisor $N'$ in $X$ with support in $D$ such that $N'\geq N$, the sheaf $\cG:= F'_{N'}R^1j_*(\Z/p^s\Z)/F'_NR^1j_*(\Z/p^s\Z)$  is a successive extension of sheaves of abelian groups which are isomorphic to vector bundles on  irreducible components of $D$. Hence $\cG\to \oplus_{\nu} i_{\nu,*}i_{\nu}^*\cG$ is injective where $\nu$ ranges over all generic points of $D$ and $i_{\nu}: \nu \to X$ is the inclusion morphism. Claim follows from this by the downward induction on the multiplicities of $N$ at generic points of $D$. 
 
 We have the homomorphism
 $$F_Nj_*W_s(\cO_U)\to \cO_X(N) \otimes_{\cO_X} \Omega^1_X(\log D)\;;\; (f_{s-1}, \dots, f_0)\mapsto -\sum_i f_i^{p^i-1}df_i.$$ 
 By the exact sequence (\ref{ASWX}) and by Claim, 
 this homomorphism induces a homomorphism
 $$F_NR^1j_*(\Z/p^s\Z) \to \cO_X(N)/\cO_X(M) \otimes_{\cO_X} \Omega^1_X(\log D)$$
 which is $\Rsw$.

 The compatibility with pullbacks (\ref{pullback}) is evident in this method.

 \end{sbpara}

As an application of our theory, we have

\begin{sbthm}\label{thmC} Let $X$ be a regular scheme of finite type over $\Z$,  $D$ a divisor on $X$ which is regular and integral,   $U=X\smallsetminus D$, $p$ a prime number, and let $\chi\in H^1_{et}(U, \Q_p/\Z_p)$. Let $x$ be a closed point of $D$. Then 
$$\Sw_D(\chi)= \sup_C   \Sw_x(\chi|_{C\cap U})/(C, D)_x$$
where $C$ ranges over all one-dimensional integral subschemes  on $X$ which contain $x$, which are regular at $x$,  and which are  not contained $D$. Here $\Sw_D(\chi)$ denotes the Swan conductor of $\chi$ at the generic point of $D$, and $(C,D)_x$ denotes the intersection number of $C$ and $D$ at $x$ which is the multiplicity at $x$ of the pullback of the divisor $D$ to $C$.  
\end{sbthm}

The positive characteristic case of \ref{thmC} is already proved by Barrientos  (\cite{Ba} Theorem 5.2) without the  assumption $X$ is of finite type over $\Z$.

 \begin{sbpara} We prove \ref{thmC}. l.h.s $\geq $ r.h.s. follows from 1 of Theorem \ref{pullback}. To prove l.h.s. $\leq$ r.h.s., we consider the $C$ in the proof of \ref{injrest} with $n=\Sw(\chi)$ (so $h=0$). We have 
$\Sw_x(\chi|_{C\cap U})= en- a$, $(C,D)_x=e$, and hence $\Sw_x(\chi|_{C\cap U})/(C,D)_x= n-e^{-1}a$. 
 $a$ can be fixed and $e$ can become arbitrarily big. 

\end{sbpara}

\begin{sbpara} 
As is discussed in \cite{Ba}, \ref{thmC} is
related to the rank one case of Conjecture A in \cite{Ba}  on ramification of  $\ell$-adic sheaves on a scheme $X$ (here $\ell$ is a prime number which is invertible on $X$). This conjecture was first formulated by Esnault and Kerz in the positive characteristic case at the end of Section 3 of \cite{EK}.  Conjecture A is formulated by using the Sawn conductor of a representation of $\Gal(\bar K/K)$ 
 for a complete discrete valuation field $K$ defined by Abbes-Saito theory \cite{AS}. For a one-dimensional Galois representation, the Swan conductor given by \cite{KKs} and used in this paper coincides with that given by \cite{AS} in the positive characteristic case (Cor. 9.12 of \cite{AS2}), but this coincidence is not yet known in the mixed characteristic case. 

If the last coincidence is true, then (1) of \ref{pullback},  \ref{thmC} and the arguments in Section 6 of \cite{Ba} show that Conjecture A  is true for rank one $\ell$-adic sheaves on $X$ in \ref{s4X}.

 \end{sbpara}

\noindent {\rm Kazuya KATO
\\
Department of mathematics
\\
University of Chicago
\\
Chicago, Illinois, 60637, USA
\\
{\tt kkato@math.uchicago.edu}

\bigskip

\noindent {\rm Isabel LEAL}
\\
 Courant Institute of Mathematical Sciences
\\
New York, N.Y., 0711, USA
\\
{\tt leal@courant.nyu.edu }

\bigskip

\noindent
{\rm Takeshi Saito}
\\
Department of Mathematical Sciences
\\
University of Tokyo
\\
Meguro-ku, Tokyo, 153-8914, Japan
\\
{\tt t-saito@ms.u-tokyo}

\end{document}